\documentclass[11pt]{article} 
\usepackage{amsmath,amsfonts,amssymb}
\usepackage{amsthm}
\usepackage{xfrac, enumitem}
\usepackage[pdftitle={Limit pretrees for free group automorphisms: existence}, pdfauthor={Jean Pierre Mutanguha}]{hyperref}

\usepackage{color}

  \theoremstyle{plain}
\newtheorem{thm}{Theorem}[section]
\newtheorem*{thm*}{Theorem}
\newtheorem*{mthm*}{Main Theorem}
\newtheorem{prop}[thm]{Proposition}
\newtheorem{cor}[thm]{Corollary}
\newtheorem{lem}[thm]{Lemma}
\newtheorem{athm}{Theorem}[section]
\newtheorem{alem}[athm]{Lemma}
\newtheorem{acor}[athm]{Corollary}
\newtheorem*{obs*}{Observation}
\newtheorem*{sum*}{Summary}

  \theoremstyle{definition}

\newtheorem*{qn*}{Question}
  \theoremstyle{remark}
\newtheorem*{rmk}{Remark}

\makeatletter
\renewcommand*\l@section[2]{%
  \ifnum \c@tocdepth >\z@
    \addpenalty\@secpenalty
    \addvspace{1.0em \@plus\p@}%
    \setlength\@tempdima{2em}%
    \begingroup
      \parindent \z@ \rightskip \@pnumwidth
      \parfillskip -\@pnumwidth
      \leavevmode \bfseries
      \advance\leftskip\@tempdima
      \hskip -\leftskip
      #1\nobreak\hfil
      \nobreak\hb@xt@\@pnumwidth{\hss #2%
                                 \kern-\p@\kern\p@}\par
    \endgroup
  \fi}
\renewcommand*\l@subsection{\@dottedtocline{2}{2em}{2.5em}}
\makeatother

\newcommand{\nonumsec}[1]{\section*{#1}
\addcontentsline{toc}{section}{\protect\numberline{}#1}}

\newcommand{\fakeenv}{} 

\newenvironment{restate}[2]  
{
  \renewcommand{\fakeenv}{#2} 
  \theoremstyle{plain}
  \newtheorem*{\fakeenv}{#1~\ref{#2}} 
  \begin{\fakeenv}
}
{
  \end{\fakeenv}
}

\newcommand*{\defeq}{\mathrel{\vcenter{\baselineskip0.5ex \lineskiplimit0pt \hbox{\scriptsize.}\hbox{\scriptsize.}}}=}

\usepackage[utf8]{inputenc} 

\usepackage{geometry} 

\usepackage{graphicx} 

\title{Limit pretrees for free group automorphisms: existence}
\author{Jean Pierre Mutanguha\thanks{{\it Email:} {\tt \href{mailto:mutanguha@mpim-bonn.mpg.de}{mutanguha@mpim-bonn.mpg.de}}, {\it Web address:} {\tt \url{https://mutanguha.com}} \newline Max Planck Institute for Mathematics, Bonn, Germany}}

\begin{document}
\maketitle

\begin{abstract} To any free group automorphism, we associate a real pretree with several nice properties. 
First, it has a rigid/non-nesting action of the free group with trivial arc stabilizers. 
Secondly, there is an expanding pretree-automorphism of the real pretree that represents the free group automorphism.
Finally and crucially, the loxodromic elements are exactly those whose (conjugacy class) length grows exponentially under iteration of the automorphism; thus, the action on the real pretree is able to detect the growth type of an element.

This construction extends the theory of metric trees that has been used to study free group automorphisms.
The new idea is that one can equivariantly blow up an isometric action on a real tree with respect to other real trees and get a rigid action on a treelike structure known as a real pretree.
Topology plays no role in this construction as all the work is done in the language of pretrees (intervals).
\end{abstract}

\renewcommand{\thefootnote}{\fnsymbol{footnote}} 
\renewcommand{\thefootnote}{\arabic{footnote}}

\section*{Introduction}

The study of free group outer automorphisms shares a lot with the theory of mapping class groups of surfaces.
The relevant dictionary replaces compact surfaces with finite graphs (i.e.~finite 1-dimensional CW-complexes) and surface homeomorphisms with graph homotopy equivalences.
While many results in surface theory have analogues in the free group setting, the situation for mapping classes tends to be comparatively well-behaved as one is working with surface homeomorphisms.
On the other hand, an infinite order outer automorphism of a free group can never be represented by a homeomorphism of a graph.
As a consequence, we have fewer tools at our disposal to study free group automorphisms.
One such important missing tool is a complete analogue of Nielsen--Thurston theory.

\subsection*{Nielsen--Thurston theory}
The main goal of the current project is defining an analogue of singular measured foliations in the free group setting.
Let us start with a summary of the Nielsen--Thurston theory for surface homeomorphisms.

Suppose~$S$ is a compact surface with negative Euler characteristic and let $f\colon S \to S$ be a homeomorphism. William Thurston proved~\cite[Theorem~4]{Thu88} there is:
\begin{enumerate}
\item a {\it canonical} (potentially empty) union~$\gamma$ of essential simple closed curves;
\item a homeomorphism~$g\colon S \to S$ isotopic to~$f$ that leaves a regular neighbourhood~$N(\gamma)$ of the multicurve $\gamma$ invariant; and
\item the restriction of $g$ to permuted components of the complement $S \setminus N(\gamma)$  is either:
\begin{itemize}
\item a {\it pseudo-Anosov} homeomorphism, i.e.~has a {\it canonical} invariant measured singular foliation whose transverse measure is scaled by a stretch factor $\lambda > 1$;~or
\item a {\it linear} homeomorphism, i.e.~reducible (possibly non-canonically) to a finite order homeomorphism --- Birman--Lubotzky--McCarthy later proved the latter reduction can be made canonical~\cite[Theorem~C]{BLM83} but that is not in the standard expositions of Thurston's result.
\end{itemize}
\end{enumerate}

Thus, up to isotopy, we may assume a given surface homeomorphism $S \to S$ has a canonical decomposition into pseudo-Anosov components and their complements, and each orbit~$S'$ of pseudo-Anosov components has an invariant measured singular foliation whose transverse measure is scaled by some stretch factor~$\lambda_{S'} > 1$. 
Lifting this decomposition to the universal cover~$\widetilde{S}$ gives a canonical partition of the cover into lifts of the foliation's leaves and the unfoliated complementary regions.

One may now endow this universal cover with a pseudo-metric using the lift of the transverse measure of the foliation.
Morgan--Shalen proved~\cite[Section~1]{MS91} that the corresponding metric space is an {\it $\mathbb R$-tree}, or simply {\it tree} as we shall call it in this paper, and the action of~$\pi_1(S)$ on the universal cover by deck transformations induces an action on this tree by isometries, or simply {\it isometric} action.
Finally, any lift of the surface homeomorphism induces an expanding ``dilation'' of the tree: the tree decomposes into finitely many orbits of subtrees and the restriction of the induced map to each subtree is an expanding homothety --- the expansion factor may vary with the subtree's orbit.

The minimal isometric $\pi_1(S)$-action on the tree has two interesting properties:
\begin{itemize}
\item arc (pointwise) stabilizers are trivial; and
\item an element of $\pi_1(S)$ is {\it elliptic} if and only if its conjugacy class is represented by a closed curve in the complement of the pseudo-Anosov components.
\end{itemize}
Moreover, the tree's equivariant dilation class is canonical by construction.

\subsection*{Our motivation and main theorem}
When we started this project, we wanted to prove a direct free group analogue of this statement about the canonical limit tree --- the dilation requirement on the expanding homeomorphism would have been relaxed and complements of pseudo-Anosov components replaced by {\it polynomially growing} subgroups.
This proved to be extremely elusive and we now doubt such a metric analogue always exists!

\medskip
Fortunately, we did manage to prove a variation to this analogy.
Let us return to the limit $\mathbb R$-tree~$T$ given by the canonical decomposition of the universal cover~$\widetilde{S}$ and the transverse measure.
Closed arcs in~$T$ determine a {\it real pretree} structure;
a real pretree is essentially what you get when you try to define an $\mathbb R$-tree without a metric or topology (see Section~\ref{subsec - trees.pretrees}).
The isometric $\pi_1(S)$-action on the $\mathbb R$-tree determines a {\it rigid} $\pi_1(S)$-action on its real pretree.
Rigid actions are commonly known as {\it non-nesting} actions in the literature~\cite{Lev98}, but we find the name ``rigid'' more evocative.



There is no simple way to capture the ``expanding'' nature of the homotheties without a metric on the real pretree.
We introduce the term {\it (free group)-expanding} to describe the induced pretree-automorphisms (see Section~\ref{subsec - trees.pretrees}).
With this new terminology, the free group analogue for the real pretree statement becomes:

\begin{mthm*}[Theorem~\ref{thm - limpretrees}] If $\phi\colon F \to F$ is an automorphism of a finitely generated free group, then there is:
\begin{enumerate}
\item a minimal rigid $F$-action on a real pretree~$T$ with trivial arc stabilizers;
\item a $\phi$-equivariant $F$-expanding pretree-automorphism $f\colon {T} \to {T}$; and
\item an element in $F$ is $T$-elliptic if and only if it \emph{grows polynomially} under $[\phi]$-iteration.
\end{enumerate}
The nontrivial point stabilizers are finitely generated. 
Moreover, these subgroups are proper and have rank strictly less than that of~$F$ if and only if~$[\phi]$ is \emph{exponentially growing}.
\end{mthm*}
\begin{rmk} 
The third condition can be restated in terms of loxodromic elements.
The actual theorem will be stated and proven in a         ``relative setting''. The growth types of elements and automorphisms are defined in Chapter~\ref{SecLimitsexp}.\end{rmk}


\medskip
We suspect that the action's rigidity can be strengthened: 
perhaps, the induced action on the {\it pretree completion} is a {\it convergence action}.
To complete the analogy with the surface setting, we discuss in the epilogue the extent to which a limit pretree produced by this theorem is canonical --- this is proven in the sequel~\cite{Mut22}.
This is of much interest as it is a canonical representation of a free group automorphism.
For instance, it allows us to classify automorphisms in terms of {\it the} limit pretree's {\it index} (see Appendix~\ref{SecIndexing}).

\medskip
In the long run, we would also like to find interesting canonical representations of polynomially growing automorphisms --- any pretree of Theorem~\ref{thm - limpretrees} is a point for these automorphisms.

\bigskip
\noindent \textbf{Proof outline for the main theorem.} Using irreducible train tracks, construct an isometric action on a limit tree with trivial arc stabilizers.
This action admits an expanding homothety that represents the given free group automorphism. 
In particular, polynomially growing elements are elliptic;
however, it is possible for some exponentially growing elements to be elliptic as well.
By considering restrictions of the automorphism to point stabilizers, we get a hierarchy of isometric actions on trees.
Most importantly, elements grow polynomially if and only if they are elliptic in each tree in the hierarchy.

The key step describes how to combine this hierarchy of trees, through blow-ups, into one {\it treelike} structure --- a real pretree.
While intuitive, the details of this construction get a bit technical.
If done appropriately, this new structure will admit an $F$-expanding pretree-automorphism.
Since we have a hierarchy of expanding homotheties, the contraction mapping theorem implies the blow-up construction can indeed be done appropriately!

\medskip
\noindent \textbf{Acknowledgments:} I would like to thank Ursula Hamenst\"adt, for suggestions that helped me clarify the blow-up construction, and the referee, for improving my exposition. I am grateful to the Max Planck Institute for Mathematics for their continued support.

\tableofcontents


\section{Preliminaries}\label{SecPrelims}

In this paper, $F$ will denote a nontrivial free group of finite rank. Note that subscripts will never indicate the rank but will instead be mostly used to index a collection of free groups. 

\subsection{Free splittings and topological representatives}
A \underline{free splitting} of~$F$ is a simplicial tree $T$ (i.e.~1-dimensional contractible CW-complex) and a {\it minimal} (left) $F$-action by simplicial automorphisms with trivial  edge stabilizers. 
We will also assume the simplicial tree~$T$ has no bivalent vertices.
Choose a maximal set of orbit representatives~$\mathcal V$ for the set of vertices of~$T$ whose {\it half-edge neighbourhood} is connected.
If the simplicial tree~$T$ is not a singleton, then the collection of nontrivial stabilizers~$\mathcal G$ of vertices in~$\mathcal V$ is a {\it proper free factor system} of~$F$ by Bass--Serre theory~\cite{Ser77}. 
After labelling the vertices of the finite quotient graph $\Gamma = F \backslash T$ with~$\mathcal G$, we get a graph of groups decomposition $(\Gamma, \mathcal G)$ of $F$ with trivial edge groups.
In fact, by the fundamental theorem of Bass--Serre theory, there is a one-to-one correspondence between free splittings of $F$ and graph of groups decompositions of $F$ with trivial edge groups.
For free splittings, we may use $(\Gamma, \mathcal G)$ as a synonym for $T$. 

\medskip
A \underline{(relative) topological representative} of an automorphism $\phi\colon F \to F$ on a free splitting $T$ of $F$ is a map $f\colon T \to T$ satisfying the following conditions:
\begin{itemize}
\item {\it cellular:} $f$ maps vertices to vertices and is injective on edges; and
\item {\it $\phi$-equivariant:} $f(x \cdot p) = \phi(x) \cdot f(p)$ for all $x \in F$ and $p \in T$.
\end{itemize}
By definition, a topological representative on $T$ induces a cellular map $[f]\colon \Gamma \to \Gamma$ on the quotient graph $\Gamma$.

Suppose $f\colon T \to T$ is a topological representative of an automorphism $\phi\colon F \to F$, $v \in \mathcal V$ a vertex orbit representative, and $G_v \in \mathcal G$ the stabilizer of $v$. 
Then $f(v) = x_v \cdot w$ for some element $x_v \in F$ and vertex orbit representative $w \in \mathcal V$. 
The stabilizer of $f(v)$ is $x_vG_w x_v^{-1}$ and, if $G_v$ is not trivial, then the restriction of $\phi$ to $G_v$ is an isomorphism $G_v \to x_vG_w x_v^{-1}$. 
By post-composing with the inner automorphism $\mathrm{inn}(x_v^{-1})\colon F \to F$ that maps $y \mapsto x_v^{-1} y x_v$, we get a homomorphism $\phi_v\colon G_v \to G_w$.
The {\it outer class} $[\phi_v]$ is independent of the chosen element~$x_v \in F$ with $f(v) = x_v \cdot w$, i.e.~the homomorphism~$\phi_v$ is unique up to post-composition with an inner automorphism of $G_w$. 
The collection $\{\phi_v~:~ v \in \mathcal V \}$ is denoted by $\left.\phi\right|_{\mathcal G}$ and called a \underline{restriction of~$\phi$} to~$\mathcal G$.

\medskip
For the rest of the paper, we shall restrict $\mathcal V$ to the subset consisting of vertices with nontrivial stabilizers.
This is mostly a stylistic choice made to simplify the exposition. 
For instance, the restriction $\left.\phi\right|_{\mathcal G}$ will permute the nontrivial stabilizers in $\mathcal G$ and, under this assumption, can be considered an automorphism of $\mathcal G$.

\subsection{Free group systems and automorphisms}
To formalize the last statement, we define a \underline{(countable) group system} $\mathcal H$ to be a disjoint union $\bigsqcup_{i \in \mathcal I} H_i$ of countably many nontrivial countable groups;
the latter are the {\it components} of $\mathcal H$. The \underline{empty system} is the group system with empty index set~$\mathcal I$.
If all component groups~$H_i$ have [property-?], then we shall call~$\mathcal H$ a ``[property-?] group system''.
For instance, we will mainly work with {\it subgroup} systems and {\it free group} systems. 
In some ambient group, a subgroup system $\mathcal G$ \underline{carries} another subgroup system $\mathcal G'$ if each $\mathcal G'$-component is contained in a conjugate of some $\mathcal G$-component.

The \underline{complexity} of a (possibly trivial) free group~$H$~is $c(H) = 2 \cdot \mathrm{rank}(H) - 1$, which takes values in $\mathbb Z_{\ge -1} \cup \{ \infty \}$.
For a (possibly empty) free group system~$\mathcal H$,
\[c(\mathcal H) \defeq \sum_{i \in \mathcal I} c(H_i) \quad \in ~ \mathbb Z_{\ge 0} \cup \{ \infty \}.\]
A group system has \underline{finite type} if the index set is finite and components are finitely generated.
In particular, a free group system has finite type exactly when its complexity is finite. In this paper, $\mathcal F$ will denote a nonempty free group system of finite type.

\medskip
The collection of nontrivial vertex stabilizers $\mathcal G$ for a free splitting of $F$ can and will be viewed as a subgroup (or rather, free factor) system of finite type (even if empty).
Similarly, we define a free splitting~$\mathcal T$ of~$\mathcal F$ to be a ``disjoint union'' of free splittings of the components of $\mathcal F$.
A free splitting~$\mathcal T$ is \underline{degenerate} if all component simplicial trees~$T_i$ are singletons.
By passing to the (representatives of) nontrivial vertex stabilizers in a nondegenerate free splitting, we can inductively form a descending chain of free factor systems with strictly decreasing complexity.
Starting with $F$, the length of such a chain is at most $c(F) + 1$.
We will mostly state and prove the results in terms of free group systems to facilitate induction on complexity.

\medskip
An \underline{automorphism} $\psi\colon \mathcal H \to \mathcal H$ of a group system~$\mathcal H$ is a disjoint union of isomorphisms $\psi_i\colon H_i \to H_{\sigma \cdot i}~(i \in \mathcal I)$ where~$\sigma \in \mathrm{Sym}(\mathcal I)$.
A subgroup system $\mathcal G = \bigsqcup_{j \in \mathcal J} G_j$ of $\mathcal H$ is \underline{$[\psi]$-invariant} if $\psi(G_j) = h_j \, G_{\alpha \cdot j} \, h_j^{-1}$ for all $j \in \mathcal J$, for some choice $( h_j \in \mathcal H : j \in \mathcal J)$ and $\alpha \in \mathrm{Sym}(\mathcal J)$. 
The collection of isomorphisms $\left.\psi\right|_{\mathcal G} = \{ \mathrm{inn}(h_j^{-1}) \circ \left.\psi\right|_{G_j} \colon G_j \to G_{\alpha \cdot j} \}$ will be called a \underline{restriction of $\psi$} to $\mathcal G$;
note that $\left.\psi\right|_{\mathcal G}$ involves an implicit choice $( h_j : j \in \mathcal J)$.

Let $\psi\colon \mathcal F \to \mathcal F$ be an automorphism and $\mathcal T$ a free splitting of $\mathcal F$. 
A topological representative of $\psi$ on $\mathcal T$ is a disjoint union of $\psi_i$-equivariant cellular maps $f_i\colon T_i \to T_{\sigma \cdot i}$.

\subsection{Train track theory}
We conclude the chapter with some preliminary results from Bestvina--Handel's theory of train tracks~\cite[Section~1]{BH92}.
A \underline{(relative) train track} for an automorphism $\psi\colon \mathcal F \to \mathcal F$ is a topological representative of $\psi$ whose iterates are all topological representatives as well; 
equivalently, a train track is a topological representative whose restrictions of iterates to any edge are all injective.
A topological representative $f \colon \mathcal T \to \mathcal T$ is \underline{irreducible} if for any pair of edges $[e], [e']$ in the quotient graph $\Gamma_* = \mathcal F \backslash \mathcal T$, there is an integer $n \ge 1$ for which $[f^n(e')]$ contains $[e]$.
To any topological representative $f$, we can associate a nonnegative integer square matrix~$A(f)$: rows and columns are indexed by the (unoriented) edges $[e], [e']$ resp. and matrix entry is the number of times $[e]$ appears in the path $[f(e')]$; then~$f$ is irreducible if and only if~$A(f)$ is {\it irreducible}.

Let $f\colon \mathcal T \to \mathcal T$ be an irreducible topological representative. 
By Perron-Frobenius theory, the matrix $A(f)$ has a unique eigenvalue $\lambda(f) \ge 1$ with a positive eigenvector $\nu(f)$;
it follows from the theory that $f$ is a simplicial automorphism if $\lambda(f) = 1$.
Using~$\nu(f)$, we can equip $\mathcal T$ with an invariant metric, i.e.~$\mathcal F$ acts by isometries with respect to this metric; 
furthermore, after applying an equivariant isotopy, the restriction of $f$ to any edge will be a $\lambda(f)$-homothety. 
The metric on $\mathcal T$ will be referred to as the \underline{eigenmetric}. 
We have set the stage for the foundational theorem due to Bestvina--Handel:

\begin{thm}[cf.~{\cite[Theorem~1.7]{BH92}}]\label{thm - findtt} If $\psi\colon \mathcal F \to \mathcal F$ is an automorphism of a free group system $\mathcal F$ and $\mathcal G$ a $[\psi]$-invariant proper free factor system of $\mathcal F$, then there is an irreducible train track $\tau\colon \mathcal T \to \mathcal T$ for $\psi$ defined on some nondegenerate free splitting $\mathcal T$ of~$\mathcal F$ whose vertex stabilizers carry $\mathcal G$.
\end{thm}

\noindent See also Rylee Lyman's \cite[Theorem~A]{Lym22} for a very general version of this theorem.

\begin{proof}[Proof outline] Section~1 of~\cite{BH92} describes an algorithm that takes a topological representative of $\phi\colon F \to F$ on a free splitting of $F$ with trivial vertex stabilizers as input and finds either an irreducible train track $\tau\colon T \to T$ on a free splitting with trivial vertex stabilizers or a topological representative on a nondegenerate free splitting with nontrivial vertex stabilizers.
However, trivial vertex stabilizers are not crucial to the algorithm and it can be adapted into one that takes a topological representative on a nondegenerate free splitting with vertex stabilizers~$\mathcal G$ as input and finds either an irreducible train track on a free splitting with the same stabilizers or a topological representative on a nondegenerate free splitting whose vertex stabilizers {\it properly carry} $\mathcal G$.

Starting with an initial input of a topological representative on a nondegenerate free splitting with vertex stabilizers~$\mathcal G$, the modified algorithm can then be repeatedly applied to its own outputs until it finds an irreducible train track on some nondegenerate free splitting whose vertex stabilizers carry~$\mathcal G$.
This process will stop after at most $(c(F)-c(\mathcal G))$ repetitions. Finally, we note that $F$ being ``connected'' was not crucial to the algorithm and~$F$ can be safely replaced with a system $\mathcal F$.
\end{proof}

Let $f\colon \mathcal T \to \mathcal T$ be a topological representative. We say an immersed path $\sigma$ in $\mathcal T$ is \underline{$f$-legal} if restrictions of $f$-iterates to $\sigma$ are all injective. 
Thus, a train track is a topological representative~$\tau$  whose edges are $\tau$-legal.
Besides knowing edges (and their forward iterates) are legal, it is useful to have \underline{legal elements}, i.e.~{\it loxodromic elements} with legal {\it axes}.

\begin{prop}\label{prop - legalaxis} If $\psi\colon \mathcal F \to \mathcal F$ is an automorphism and $\tau\colon \mathcal T \to \mathcal T$ an irreducible train track for $\psi$, then 
each nondegenerate component of~$\mathcal T$ contains a $\tau$-legal axis of some loxodromic element of~$\mathcal F$.
\end{prop}

\begin{proof} If $\lambda(\tau)=1$, then $\tau$ is a simplicial automorphism and all immersed paths in~$\mathcal T$ are $\tau$-legal. 
So we may assume $\lambda(\tau)>1$. 
Choose an oriented edge~$e_i$ in a component~$T_i \subset \mathcal T$. 
Irreducibility of train track~$\tau$ and $\lambda(\tau)>1$ imply we can find distinct translates $x_1 \cdot e_i, x_2 \cdot e_i$ in the oriented edge-path~$\tau^n(e_i)$ for some large~$n$.
The $\mathcal T$-loxodromic element~$x_2x_1^{-1}$ has a $\tau$-legal axis in~$\mathcal T$: a fundamental domain for the axis is the $\tau$-legal subpath of~$\tau^{n}(e_i)$ joining the midpoints of the chosen two translates of~$e_i$.
\end{proof}

\section{Trees and real pretrees}\label{SecTrees}

Expanding irreducible train tracks are very useful for understanding the dynamics of an automorphism since the iterates of edges all expand by the same factor under the eigenmetric.
Unfortunately, since the train track necessarily fails to be injective near some vertices, not all paths will similarly expand: for instance, there are so-called ``periodic Nielsen paths'' whose length (after reduction) remains uniformly bounded under iteration.

One way to get around this is to promote the expanding irreducible train track to an expanding homothety; however, this promotion requires leaving the category of simplicial trees and working in the metric category instead (Proposition~\ref{prop - topforest}).

\subsection{Trees and index theory}\label{subsec - trees.metric}
A \underline{(metric) tree} is a 0-hyperbolic geodesic metric (nonempty) space --- 0-hyperbolic means the union of any two sides of a geodesic triangle contains the third side of the triangle;
this is also known as an {\it $\mathbb R$-tree}.
More generally, a \underline{forest} $\mathcal T$ is a disjoint union $\bigsqcup_{i \in \mathcal I} T_i$ of trees. 
Let~$T$ be a tree; a \underline{direction} at a point $p \in T$ is a component of the complement~$T \setminus \{ p \}$. 
A \underline{branch point} of $T$ is a point with at least three directions. 

A $\lambda$-homothety $f\colon T \to T$ is \underline{expanding} if $\lambda > 1$. 
Being expanding is invariant under iteration and composition with an isometry. 
It follows from the contraction mapping theorem that an expanding $\lambda$-homothety of a {\it complete} tree has a unique fixed point and it is a {\it repellor}. Generally, a \underline{$\lambda$-homothety} $f\colon \mathcal T \to \mathcal T$ of a forest $\mathcal T = \bigsqcup_{i \in \mathcal I} T_i$ is a disjoint union of $\lambda$-homotheties $f_i\colon T_i \to T_{\sigma \cdot i}~(i\in \mathcal I)$, where $\sigma \in \mathrm{Sym}(\mathcal I)$.

An isometry of a tree is either \underline{elliptic} if it fixes a point or \underline{loxodromic} if it has a unique minimal invariant subtree isometric to $\mathbb R$, known as the {\it axis}, on which its acts by a nontrivial translation. 
We will mostly only care about {\it isometric actions} on trees with {\it finite arc stabilizers}, i.e.~actions by isometries where the stabilizers of nondegenerate intervals are finite. 
When the acting group is torsion-free, we can equivalently say isometric actions with {\it trivial arc stabilizers}.

Suppose $d_T$ is the metric on $T$; an isometry $\iota\colon T \to T$ has \underline{translation distance} given by \[\| \iota \| = \inf_{p \in T} d_T(p, \iota(p)).\]
Any isometric $H$-action on a tree has a nonnegative function $\|\cdot \|\colon H \to \mathbb R_{\ge 0}$ given by the translation distances. 
For a group system~$\mathcal H = \bigsqcup_{i \in \mathcal I} H_i$ and a forest~$\mathcal T = \bigsqcup_{i \in \mathcal I} T_i$, an isometric $\mathcal H$-action on $\mathcal T$ is a disjoint union of isometric $H_i$-actions on~$T_i$.

A \underline{characteristic subtree} for an isometric action on a tree~$T$ is a minimal invariant subtree in~$T$. 
An isometric action on a tree~$T$ is \underline{minimal} if~$T$ is the characteristic subtree.

\medskip
Suppose $F$ acts minimally on a nondegenerate tree~$T$ by isometries with trivial arc stabilizers. 
We now introduce the {\it index theory} for such actions.
For each $F$-orbit of points $[p] \in F \backslash T$, let $G_p \le F$ denote the stabilizer of~$p \in T$ and~$\#\mathrm{dir}[p]$ denote the number of $G_p$-orbits of directions at $p$.
The (local) {\it index} at~$[p]$~is:
\[i[p] \defeq c(G_p) - 1~ +~\#\mathrm{dir}[p] \quad \in  ~ \mathbb Z_{\ge 0} \cup \{ \infty \}.\]
The (global) \underline{index} of $F \backslash T$ is:
\[i(F \backslash T) \defeq \sum_{[p] \, \in \, F \backslash T} i[p]. \]
One of our main tools will be Gaboriau--Levitt's \underline{index inequality}~\cite[Theorem~III.2]{GL95}:
\[i(F \backslash T) < c(F).\]
For example, the index of any free splitting of~$F$ is $c(F)-1$.
In Appendix~\ref{SecIndexing}, we will sketch Gaboriau--Levitt's proof in a metric-agnostic setting. 
Corollaries of the inequality: there are finitely many $F$-orbits of directions at branch points and the point stabilizers subgroup system is represented by a system that has strictly lower complexity than~$F$.
For a free group system of finite type, the index of its isometric action on a forest is the sum of indices for each component and the index inequality still holds.

\subsection{Pretrees and rigid actions}\label{subsec - trees.pretrees}
Brian Bowditch's paper~\cite{Bow99tree} is a good survey about pretrees and other ``treelike'' structures.
A \underline{pretree} is a (nonempty) set~$T$ and a function $[\cdot,\cdot]\colon T \times T \to \mathcal P(T)$, i.e.~an association of a subset $[p,q] \subset T$ to each pair of points $p,q \in T$, that satisfies the {\it pretree axioms}: 
for all $p,q,r \in T$,
\begin{enumerate}
\item \label{axiom - symmetry} (symmetric) $[p,q] = [q,p]$ contains $\{p, q\}$;
\item \label{axiom - thin} (thin) $[p,r] \subset [p,q] \cup [q,r]$; and
\item \label{axiom - linear} (linear) if $r \in [p,q]$ and $q \in [p,r]$, then $q = r$.
\end{enumerate}
The subsets $[p,q]$ from the definition will be refered to as {\it closed intervals} and should be thought of as encoding a ``betweenness'' relation on $T$. 
Define the {\it open interval} $(p,q)$ to be the subset $[p,q] \setminus \{p,q\}$. 
Similarly, define {\it half-open intervals} $[p,q) = (q,p] = [p,q] \setminus \{q\}$.
Naturally, the real line $\mathbb R$ is a pretree; 
a tree has a \underline{(canonical) pretree structure} where closed intervals are the closed geodesic segments;
also, any subset~$S \subset T$ of a pretree inherits a pretree structure given by~$[\cdot,\cdot]_S \defeq [\cdot, \cdot] \cap S$.

A \underline{direction} at $p \in T$ is a maximal subset $D_p \subset T$ for which $p \notin [q,r]$ for all $q, r \in D_p$. 
As we did for trees, a \underline{branch point} for a pretree is a point with at least three directions.
The {\it observers' topology} on a pretree is the canonical topology generated by the subbasis of directions; Bowditch calls it the {\it order topology}~\cite[Section~7]{Bow99tree}.
Generally, a direction at a nonempty subset $S \subset T$ is a maximal $D_S \subset T$ for which $[q,r] \cap S = \emptyset$ for all $q,r \in D_S$.

\medskip
Suppose~$T$ and~$T'$ have pretree structures $[\cdot, \cdot]$ and $[\cdot, \cdot]'$ respectively. 
A set-bijection $f\colon T\to T'$ is a \underline{pretree-isomorphism} if $[f(p), f(q)]' = f([p,q])$ for all $p,q \in T$.
An injection $f\colon T \to T'$ is a \underline{pretree-embedding} if it is a pretree-isomorphism onto the image~$f(T)$ with the inherited pretree structure;
we will only need pretree-embeddings for Appendix~\ref{SecIndexing}.
Note that pretree-isomorphisms induce homeomorphisms of the observers' topologies.
There is another canonical topology finer than the observers' topology (see the interlude chapter);
we will not need either canonical topologies for the results of the paper. 

For a pretree-automorphism $f \colon T \to T$, the fixed-point set $\mathrm{Fix}_T(f)$ is the subset of points in~$T$ fixed by~$f$.
A pretree-automorphism $f$ is \underline{rigid} if either $\mathrm{Fix}_T(f)$ is empty or~$f$ fixes no direction at $\mathrm{Fix}_T(f)$.
An isometry of a (subset of a) tree is a rigid pretree-automorphism of its (inherited) pretree structure.
A subset $C \subset T$ is \underline{convex} if $[p,q] \subset  C$ for all $p,q \in C$.
The fixed-point set $\mathrm{Fix}_T(f)$ of a rigid pretree-automorphism is convex.

\medskip
A pretree is \underline{real} if every closed interval $[p,q]$ is pretree-isomorphic to a closed interval of~$\mathbb R$.
The pretree structure of a tree is real.
A nondegenerate convex subset $A \subset T$ is an \underline{arc} if any $x,y,z \in A$ simultaneously lie in  some closed interval.
A pretree is \underline{complete} if every arc is an interval.
Notably, any real pretree~$T$ embeds in a canonical complete pretree~$\widehat T$ known as the \underline{pretree completion}~\cite[Lemma~7.14]{Bow99tree}.
A real pretree is \underline{short} if every arc is pretree-isomorphic to an arc in~$\mathbb R$.
The real pretree structure of a tree is short.
For any short real pretree~$T$, the pretree completion~$\widehat T$ is real~\cite[Lemma~7.15]{Bow99tree}.
For example, the pretree $\mathbb R$ is not complete (it is an arc but not an interval), but it is short by definition; 
the pretree completion of $\mathbb R$ is the {\it extended real line}, which has two additional points $\{ \pm \infty \}$, may be denoted $[-\infty, \infty]$, and is pretree-isomorphic to a closed interval of~$\mathbb R$.
The {\it long line} is the prototype of a real pretree whose pretree completion, the {\it extended long line}, is not real: it is a closed interval not pretree-isomorphic to a closed interval of~$\mathbb R$.

\medskip
A pretree-automorphism~$f\colon T \to T$ of a real pretree is \underline{elliptic} if it has a fixed point and \underline{loxodromic} otherwise;
in the latter case, there is a maximal arc~$A \subset T$, known as the {\it axis}, preserved by~$f$.
For an elliptic pretree-automorphism~$f$ of a real pretree~$T$, the complement $T \setminus \mathrm{Fix}_T(f)$ is ``open'' in the following sense: the complement $(p,q) \setminus \mathrm{Fix}_T(f)$ is a union of open intervals for any $p,q \in T$;
in particular, a direction~$d$ at $\mathrm{Fix}_T(f)$ has an {\it attaching} point $p_d \in \mathrm{Fix}_T(f)$.
If~$f$ is also rigid, then the direction~$d$ has a unique attaching point due to convexity.
In the literature, rigid pretree-automorphisms of real pretrees have been studied as {\it non-nesting homeomorphisms}~\cite{Lev98} --- Levitt's results are stated with $\mathbb R$-trees but the metrics are never used and so the results apply to pretrees as well.
Again, we mostly only care about {\it rigid actions} on real pretrees with {\it finite arc stabilizers}, i.e.~actions by rigid pretree-automorphisms where the (pointwise) stabilizers of arcs are finite.

A \underline{characteristic convex subset} for a rigid action on a real pretree~$T$ is a minimal invariant nonempty convex subset of~$T$.
A rigid action on a real pretree~$T$ is \underline{minimal} if~$T$ is the characteristic convex subset.
Note that a real pretree~$T$ that admits a minimal rigid action by a countable group must be short: $T$ will be a countable union of closed intervals; so its arcs are countable ascending unions of closed intervals and hence pretree-isomorphic to arcs in~$\mathbb R$. 

\medskip
Suppose we have a minimal rigid $F$-action on a nondegenerate real pretree~$T$ with trivial arc stabilizers.
Then we can define the \underline{index} of $F \backslash T$ exactly as we did for minimal isometric $F$-actions with trivial arc stabilizers. Gaboriau--Levitt's index inequality still holds since their proof extends almost verbatim to this setting of minimal rigid actions with trivial arc stabilizers (see Appendix~\ref{SecIndexing} for the sketch).

\medskip
We now introduce a new term that will be our replacement for expanding homotheties in the real pretree setting.
Let $\phi\colon F \to F$ be an automorphism and $T$ a real pretree with a chosen minimal rigid $F$-action whose arc stabilizers are trivial.
Recall that $f\colon T \to T$ is $\phi$-equivariant if $f(x \cdot p) = \phi(x) \cdot f(p)$ for all $x \in F$ and $p \in T$.
A $\phi$-equivariant pretree-automorphism $f\colon T \to T$ \underline{expands} at $p \in \mathrm{Fix}_T(f)$ if each orbit~$G_p \cdot d$ of directions at~$p$ contains some half-open interval $(p, q_d]$ and each $(p,q_d]$ properly embeds in a $G_p$-translate of $f((p, q_{d'}])$ for some orbit~$G_p \cdot d'$ of directions at~$p$.
By the index inequality, there are finitely many $G_p$-orbits of directions at~$p$;
so for some $n \ge 1$, each $(p, q_d]$ properly embeds in a $G_p$-translate of $f^n((p, q_d])$ --- this justifies the term ``expanding''.

A $\phi$-equivariant pretree-automorphism $f\colon T \to T$ is \underline{$F$-expanding} if: for any element $x \in F$ and integer $n \ge 1$, the composition $x \circ f^n \colon T \to T$ is either: 1) elliptic with a unique fixed point and it expands at this fixed point; or 2) loxodromic with an axis that is not shared with any $T$-loxodromic element in~$F$.
For the motivation behind this definition, note that $\phi$-equivariant expanding homotheties are (by the contraction mapping theorem) $F$-expanding pretree-automorphism of the canonical pretree structure.





\nonumsec{Interlude --- same trees, different views}\label{SecTreetops}

This interlude is meant to describe the different perspectives on ``simplicial trees.''
As it was borne out of my own failure to appreciate the differences before starting this project, I am writing the interlude from a personal point of view.

What is a {\it simplicial tree}? 
Figure~\ref{fig - tree} is visual representation of a simplicial tree that will be the \underline{running example} for the interlude.

\begin{figure}[ht]
 \centering 
 \includegraphics[scale=1]{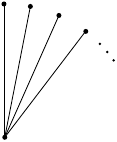}
 \caption{A tree with one ``branch point'' and (countably) infinitely many edges attached.}
 \label{fig - tree}
\end{figure}

The most elementary definition describes it as a combinatorial object.
A {\it simple graph} is a pair $(V, E)$ where~$V$ is the vertex set and the edge set~$E$ is a collection of size~$2$ subsets of~$V$. 
I could use this to define {\it reduced paths} and {\it cycles}, then a \underline{simple tree} would be a path-connected cycle-free simple graph.
For example, this was the language used by Serre in~\cite{Ser77} and it is the quickest way to define {\it free splittings}.
As a simple tree, the running example can be defined as: $V = \mathbb Z_{\ge 0}$ and $E = \{\,\{ 0, n\}~:~n \ge 1\,\}$.
This definition is unique up to a simplicial automorphism: a set-bijection of vertex sets that induces a set-bijection of the edge sets.

This perspective is especially useful if you are interested in algorithmic questions.
The downside is that it can get quite cumbersome to describe maps between simple trees that send edges to paths.
Instead, it helps to work with topological spaces where the language of maps and their deformations is already well-established.

\medskip
The first topological definition: a \underline{cellular tree} is a 1-dimensional contractible CW-complex.
This was the parenthetical definition I gave at the start of Chapter~\ref{SecPrelims}.
This definition assumes you already know what a CW-complex is.
For our purposes, a 1-dimension CW-complex is a quotient of the disjoint union of a discrete space (vertex set) and copies of closed intervals $[0,1]$ with an equivalence relation identifying each endpoint of the closed intervals with some vertex.
The complex is endowed with the quotient topology and
I will skip defining contractibility.
If you were a pedantic reader, then you may have noticed that I abused terminology while defining free splittings: I define a simplicial tree as a ``cellular tree'' yet I require that the free group act by ``simplicial automorphism'';
I never exactly explained what a simplicial automorphism of a CW-complex is!
The moral of the story is that I have prioritized brevity.
Anyway, as a cellular tree, the running example is defined as the quotient of $\{ v \} \sqcup \left( \mathbb Z_{\ge 1} \times [0,1]\right)$ with the equivalence relation generated by $v \sim (n, 0)$ for all~$n \ge 1$.
The set of equivalence classes (without the topology) will be denoted~$T$ and let $\pi$ be the quotient function onto $T$.

Now that I am dealing with topological spaces, I can discuss concepts like convergence or compactness. 
Let $x_n = (1, \frac{1}{n}), y_n = (n, \frac{1}{n}),$ and $z_n = (n, 1)$.
Then the sequence $(x_n)_{n\ge 1}$ converges to $\pi(v)$ in the CW/quotient topology, while the sets $\{ y_n \}_{n \ge 1}$ and $\{ z_n \}_{n \ge 1}$ have no limit points; moreover, this topology is not metrizable.

\medskip
The metric definition: a \underline{discrete (metric) tree} is a (metric) tree that is the convex hull of its singular points (i.e.~number of directions at the point is not 2) and the subspace of these points is discrete.
Metric spaces have a topology generated by the basis of open balls.
In practice, this topology is secondary and it is usually more convenient to work with the metric directly.
To view~$T$ as a discrete tree, equip it with the combinatorial {\it convex} metric:
\[ d_T(\pi(n,a), \pi(m,b) ) =
\begin{cases} |a-b| & \text{if } m = n \\
a + b & \text{otherwise}.
\end{cases} \]
The metric topology on $T$ is {\it strictly coarser} than the CW-topology as it contains fewer open sets.
This time, the sequences $(x_n)_{n\ge 1}$ and $(y_n)_{n\ge 1}$ both converge to $\pi(v)$ in the metric topology.
On the other hand, the set $\{ z_n \}_{n \ge 1}$ still has no limit points.
So the metric topology is not compact.

As you shall see in the next chapter, the metric setting helps us understand what an automorphism does under iteration, especially at the ``limit''. 
This starts with equipping free splittings (simple or cellular trees) with the eigenmetric of an irreducible train track and taking the {\it projective-limit} of iterating the train track to get a useful and probably nondiscrete tree.

\medskip
The final topological definition: a \underline{(separable) ``visual'' tree} is a connected subspace of a dendrite (i.e.~compact separable metric tree) that is the convex hull of its own singular points and whose branch points cannot accumulate to a branch point along one direction.
I could remove the separability condition by replacing `dendrite' with `dendron,' but that would necessitate a definition for dendrons.
There is no metric characterization for dendrons as there is for dendrites and, unfortunately,
their purely topological definition is beyond the scope of this paper; see~\cite{Bow99tree} for details.
A pretree characterization for dendrons can be: the observers' topology on a complete real pretree.
As a visual tree, the running example is the observers' topology on the cellular or metric tree $T$.
The observers' topology is even coarser than the metric topology and it makes the sequences $(x_n)_{n\ge 1}$, $(y_n)_{n\ge 1}$, and $(z_n)_{n\ge 1}$ all converge to $\pi(v)$.
In fact, $T$ with the observers' topology is compact --- it is a dendrite!
As a visual tree, the topology is metrizable and one compatible convex metric is given by 
\[ d_T'(\pi(m,a), \pi(n,b) ) =
\begin{cases} \frac{1}{m}|a-b| & \text{if } m = n \\
\frac{a}{m} + \frac{b}{n} & \text{otherwise}.
\end{cases} \]

The underlying thesis of this project is that the ``metric category'' might not be right for defining {\it blow-ups} of nondiscrete trees.
It is not clear to me that it is even possible in general! However, defining blow-ups in the ``visual category'' is rather natural.
In fact, we avoid topology altogether and carry out the construction in the ``pretree category''!

\medskip
Finally, a topological-combinatorial hybrid definition: a {\it simple pretree} is a pretree whose closed intervals are finite sets and a \underline{simple real pretree} is real pretree that is the convex hull of its singular points and these points form a simple pretree. 
The simple pretree definition is simply paraphrasing that of a simple tree, while the simple real pretree definition simultaneously captures the combinatorial nature of trees and allows topological approaches.
Simple real pretrees have two canonical topologies: the coarser one is the observers' topology; the finer one is more-or-less the CW-topology.

\medskip
So what is a simplicial tree? 
Well, it could be a simple tree, cellular tree, discrete tree, visual tree, or simple real pretree;
the answer depends on what I need it for.
For instance, if I want to discuss ``simplicial actions'' on~$\mathbb R$, then only the first two definitions are applicable --- $\mathbb R$ has no singular points!

\section{Limit pretrees: exponentially growing automorphisms}\label{SecLimitsexp}

Returning to free group automorphisms, we can now use translation distances in free splittings (with invariant combinatorial metrics) to classify free group automorphisms.  

Fix an automorphism $\psi\colon\mathcal F \to \mathcal F$ and a $[\psi]$-invariant proper free factor system~$\mathcal G$ of $\mathcal F$.
By Theorem~\ref{thm - findtt}, there is an irreducible train track $\tau^{(1)}$ for $\psi$ on a free splitting $(\Gamma_*^{(1)}, \mathcal F_2)$ of $\mathcal F_1 = \mathcal F$, where $\mathcal F_2$ carries $\mathcal G$. 
The train track~$\tau^{(1)}$ determines a restriction~$\left.\psi\right|_{\mathcal F_2}$;
applying the theorem repeatedly, we get a finite {\it descending sequence} of irreducible train tracks $\tau^{(i)}$ for restrictions $\left.\psi\right|_{\mathcal F_{i}}$ on free splittings $(\Gamma_*^{(i)}, \mathcal F_{i+1})$ of $\mathcal F_{i}$ and the final train track $\tau^{(k)}$ in the sequence is defined rel.~$\mathcal G$, i.e.~$\mathcal F_{k+1} = \mathcal G$. 
The next proposition collects standard facts that immediately follow from this train track theory;
we sketch the proof mainly to highlight the study of automorphisms through the {\it blow-up of a free splitting rel.~another free splitting}.

\begin{prop}\label{prop - growth} Let $\psi\colon\mathcal F \to \mathcal F$ be an automorphism, $\mathcal G$ a $[\psi]$-invariant proper free factor system of $\mathcal F$, and $(\Lambda_*, \mathcal G)$ a free splitting of~$\mathcal F$. 
For any element~$x$ in~$\mathcal F$, the limit inferior of the sequence~$\left(\sqrt[n]{\|\psi^n(x)\|_{\Lambda*}} \right)_{n \ge 1}$ is finite and independent of $(\Lambda_*, \mathcal G)$.

Let $\left(\tau^{(i)}\right)_{i=1}^k$ be a descending sequence of irreducible train tracks with~$\tau^{(k)}$ defined rel.~$\mathcal G$. Then the following are equivalent: 
\begin{enumerate}
\item \label{cond - growth.tt} $\lambda\left(\tau^{(i)}\right) = 1$ for all $i = 1, \ldots, k$;
\item \label{cond - growth.seq} for any $x \in \mathcal F$, the sequence~$\left(\|\psi^n(x)\|_{\Lambda*} \right)_{n \ge 1}$ is bounded by a 
polynomial in~$n$; and
\item \label{cond - growth.exprate} $\underset{n\to \infty}\liminf \sqrt[n]{\|\psi^n(x)\|_{\Lambda*}} \le 1$ for all $x \in \mathcal F$.
\end{enumerate}
\end{prop}

\noindent An automorphism is \underline{polynomially growing rel.~$\mathcal G$} if these conditions hold; otherwise,
it is \underline{exponentially growing rel.~$\mathcal G$}. 

\begin{proof}[Sketch of proof]
The automorphism~$\psi$ has a Lipschitz topological representative on the free splitting $(\Lambda_*, \mathcal G)$ of~$\mathcal F$.
So for any element~$x$ in $\mathcal F$, we can set
$
\underline{\lambda_x} \defeq \underset{n\to \infty}\liminf \sqrt[n]{\|\psi^n(x)\|_{\Lambda*}}$ in~$\mathbb R_{\ge 0}$.
By $[\psi]$-invariance of~$\mathcal G$, $\underline{\lambda_x} = 0$ if $\| x \|_{\Lambda_*} = 0$.
Since~$0$ is isolated in the image of~$\| \cdot \|_{\Lambda_*}$, we get $\underline{\lambda_x} \ge 1$ if $\| x \|_{\Lambda_*} > 0$.
Any pair of free splittings $(\Lambda_*, \mathcal G)$ and $(\Omega_*, \mathcal G)$ of~$\mathcal F$ have equivariant Lipschitz maps between them and so their translation distance functions are comparable, i.e.~
\[\frac{1}{K} \| \cdot \|_{\Lambda_*} \le \| \cdot \|_{\Omega_*} \le K \| \cdot \|_{\Lambda_*}  \text{\,(pointwise) for some constant } K \ge 1.\]
This implies~$\underline{\lambda_x}$ is independent of the free splitting~$(\Lambda_*, \mathcal G)$ and concludes the first part.

\medskip
$(\ref{cond - growth.tt} \implies \ref{cond - growth.seq})$: 
Suppose $\lambda(\tau^{(i)}) = 1$ for all $i = 1, \ldots, k$. 
So the train tracks $\tau^{(i)}$ are simplicial automorphisms;
in particular, all edge-paths in $(\Omega_*^{(k)}, \mathcal G) = (\Gamma_*^{(k)}, \mathcal G)$ have constant growth.
For induction, assume all edge-paths in $(\Omega_*^{(i)}, \mathcal G)$ have at most degree $(k-i)$ polynomial growth.
Let $(\Omega_*^{(i-1)}, \mathcal G)$ be a blow-up of $(\Gamma_*^{(i-1)},\mathcal F_{i})$ with respect to $(\Omega_*^{(i)}, \mathcal G)$.
As $\tau^{(i-1)}$ is a simplicial automorphism, each ``top-stratum'' edge gains a predetermined prefix and suffix in $(\Omega_*^{(i)}, \mathcal G)$ under iteration; 
these prefixes and suffixes have at most degree $(k-i)$ polynomial growth.
So edge-paths in $(\Omega_*^{(i-1)}, \mathcal G)$ have at most degree $(k-i+1)$ polynomial growth.
By induction and comparability again, edge-paths in $(\Lambda_*, \mathcal G)$ have at most degree $(k-1)$ polynomial growth.

$(\ref{cond - growth.seq} \implies \ref{cond - growth.exprate})$: Polynomials are sub-exponential.

$(\ref{cond - growth.exprate} \implies \ref{cond - growth.tt})$: Suppose $\lambda = \lambda\left(\tau^{(i)}\right) > 1$ for some~$i$.
Equip $(\Gamma_*^{(i)}, \mathcal F_{i+1})$ with the eigenmetric~$d_i$ and let $x \in \mathcal F_i$ be a $\tau^{(i)}$-legal element (Proposition~\ref{prop - legalaxis}).
The choice of metric and legality of~$x$ imply $\|\psi^n(x)\|_{d_i} = \lambda^n \cdot \| x \|_{d_i} > 0$.
Let~$(\Omega_*^{(1)}, \mathcal G)$ be the free splitting constructed above.
Then $\| \cdot \|_{d_i}$ and~$\| \cdot \|_{\Gamma_*^{(i)}}$ are comparable, $\|\cdot \|_{\Gamma_*^{(i)}} \le \|\cdot \|_{\Omega_*^{(i)}}$\,pointwise, and~$\| \cdot \|_{\Omega_*^{(i)}}$ is the restriction of~$\| \cdot \|_{\Omega_*^{(1)}}$.
So $\underline{\lambda_x} \ge \lambda > 1$.
%
%
\end{proof}

For any element~$x$ in~$\mathcal F$, define $\underline \lambda(x; \psi, \mathcal G) \defeq \underset{n\to \infty}\liminf \sqrt[n]{\|\psi^n(x)\|_{\mathcal G}}$. 
We say $x \in \mathcal F$ grows \underline{exponentially rel.~$\mathcal G$} if $\underline \lambda(x;\psi,\mathcal G) > 1$; otherwise, it grows \underline{polynomially rel.~$\mathcal G$}.
The ``rel.~$\mathcal G$'' will be omitted when~$\mathcal G$ is empty.
Suppose~$\mathcal H$ is a $[\psi]$-invariant subgroup system of finite type that carries~$\mathcal G$.
By passing to the characteristic subforest for~$\mathcal H$ in a free splitting $(\Gamma_*,\mathcal G)$ of~$\mathcal F$, we see that $\underline \lambda(x;\left.\phi\right|_{\mathcal H},\mathcal G) = \underline \lambda(x;\psi,\mathcal G)$ for elements~$x$ in~$\mathcal H$.
We use this observation whenever we pass to invariant subgroup systems of finite type.

\begin{rmk}
It was known since the introduction of train tracks that the first sequence in the statement of Proposition~\ref{prop - growth} converges (see~\cite[Remark~1.8]{BH92}).
Gilbert Levitt gave a finer classification for an element's growth rates~\cite[Theorem 6.2]{Lev09}.
\end{rmk}

The next proposition appears in Levitt--Lustig's paper~\cite[Proposition~3.2]{LL00}. 
Our proof takes a slightly different approach using the {\it blow-up of a free splitting rel.~expanding forest}.
This is done to highlight ideas crucial to the proof of the main theorem.

\begin{prop}\label{prop - topforest} 
Let $\psi\colon\mathcal F \to \mathcal F$ be an automorphism and $\mathcal G$ a $[\psi]$-invariant proper free factor system of $\mathcal F$. 
If $\psi\colon \mathcal F \to \mathcal F$ is exponentially growing rel.~$\mathcal G$, then there is:
a minimal isometric $\mathcal F$-action on a nondegenerate forest $\mathcal Y$ with trivial arc stabilizers for which~$\mathcal G$ is elliptic; and a $\psi$-equivariant expanding homothety $h\colon \mathcal Y \to \mathcal Y$.

In particular, $\mathcal Y$-loxodromic elements in~$\mathcal F$ grow exponentially rel.~$\mathcal G$.
\end{prop}

\begin{proof}[Proof (modulo a black box)]
Suppose $\psi\colon\mathcal F \to \mathcal F$ is exponentially growing rel.~$\mathcal G$ and set $\mathcal F_1 = \mathcal F$. 
Then there is a finite sequence of irreducible train tracks~$\tau^{(i)}$ for restrictions $\left.\psi\right|_{\mathcal F_{i}}$ on nondegenerate free splittings $(\Gamma_*^{(i)}, \mathcal F_{i+1})$ of~$\mathcal F_{i}$ with $\lambda(\tau^{(i)}) > 1$ for some $i \ge 1$; 
we equip the free splittings with eigenmetrics~$d_i$.
After truncating the sequence if necessary, assume $\lambda(\tau^{(i)}) = 1$ for $i = 1, \ldots, k-1$ and $\lambda = \lambda(\tau^{(k)}) > 1$.
By Proposition~\ref{prop - legalaxis}, 
there is a $\tau^{(k)}$-legal element~$x_0$ in~$\mathcal F_k$.

As~$\tau^{(k)}$ is $\lambda$-Lipschitz, we can define a limit function 
\[\|\cdot \|^{(k)}\colon\mathcal F_{k} \to \mathbb R_{\ge 0} \quad \text{by} \quad x \mapsto \underset{n \to \infty} \lim \lambda^{-n} \|\psi^n(x)\|_{\Gamma_*^{(k)}}.\]
Culler--Morgan proved $\|\cdot\|^{(k)}$ is a translation distance function for a minimal isometric $\mathcal F_{k}$-action on a forest $\mathcal T^{(k)}$ with {\it cyclic arc stabilizers}~\cite[Theorem~5.3]{CM87};
the forest is not degenerate as $\|x_0\|^{(k)} = \|x_0\|_{\Gamma_v^{(k)}} > 0$.
Lustig proved the $\mathcal F_{k}$-action on~$\mathcal T^{(k)}$ has trivial arc stabilizers~\cite[Appendix]{GLL98}. 
There is a $\left.\psi\right|_{\mathcal F_{k}}$-equivariant $\lambda$-homothety $h^{(k)}\colon \mathcal T^{(k)} \to \mathcal T^{(k)}$ since $\|\psi(x) \|^{(k)} = \lambda \|x\|^{(k)}$ for all elements~$x$ in~$\mathcal F_{k}$~\cite[Theorem~3.7]{CM87}.
By construction, the $[\psi]$-invariant proper free factor system~$\mathcal G$ is $\mathcal T^{(k)}$-elliptic.

\medskip

The main step of the proof is a construction that allows us to ``merge'' the higher {\it polynomial strata} and the {\it exponential stratum} at $k$. 
If $k = 1$, then set $\mathcal Y = \mathcal T^{(k)}$, $h = h^{(k)}$, and we are done. 
Otherwise, $k \ge 2$ and, for some $i \le k$, there is:
a minimal isometric $\mathcal F_{i}$-action on a nondegenerate forest~$\mathcal T^{(i)}$ with trivial arc stabilizers and an equivariant copy of~$\mathcal T^{(k)}$; and
a $\left.\psi\right|_{\mathcal F_{i}}$-equivariant $\lambda$-homothety $h^{(i)}\colon \mathcal T^{(i)} \to \mathcal T^{(i)}$.

In the next chapter (simple patchwork~\ref{thm - simple patch}, black box), we construct a unique forest~$\mathcal T^{(i-1)}$ that is ``an equivariant blow-up and simplicial collapse'' of $(\Gamma_*^{(i-1)}, \mathcal F_{i})$ with respect to $\tau^{(i-1)}$ and~$h^{(i)}$. In particular, there is: a minimal isometric $\mathcal F_{i-1}$-action on $\mathcal T^{(i-1)}$ with trivial arc stabilizers and an equivariant copy of~$\mathcal T^{(i)}$; and
a $\left.\psi\right|_{\mathcal F_{i-1}}$-equivariant $\lambda$-homothety $h^{(i-1)}$ on~$\mathcal T^{(i-1)}$ induced by $\tau^{(i-1)}$ and~$h^{(i)}$.
By induction, we have: a minimal isometric $\mathcal F$-action on a forest $\mathcal Y = \mathcal T^{(1)}$ with trivial arc stabilizers, an equivariant copy of~$\mathcal T^{(k)}$, and a $\psi$-equivariant $\lambda$-homothety~$h=h^{(1)}$ on~$\mathcal Y$.

\medskip
For the last part of the proposition, choose a free splitting $(\Lambda_*, \mathcal G)$ for $\mathcal F$. 
Let $\|\cdot\|_{\Lambda_*}$ and $\| \cdot \|_{\mathcal Y}$ be the translation distance functions for $(\Lambda_*, \mathcal G)$ and~$\mathcal Y$ respectively.
Since~$\mathcal G$ is $\mathcal Y$-elliptic, there is an equivariant Lipschitz map $(\Lambda_*, \mathcal G) \to \mathcal Y$ and \[ \|\cdot \|_{\mathcal Y} \le K \| \cdot \|_{\Lambda_*} \text{\,pointwise for some constant } K \ge 1.\]
Therefore, $\lambda^n \| \cdot \|_{\mathcal Y} = \|\psi^n(\cdot)\|_{\mathcal Y} \le K \|\psi^n(\cdot)\|_{\Lambda_*}$. 
So $\mathcal Y$-loxodromic elements in~$\mathcal F$ grow exponentially rel.~$\mathcal G$.
\end{proof}

The main theorem of this paper associates to a free group automorphism a real pretree with a minimal rigid $F$-action whose loxodromic elements are precisely the elements that grow exponentially. 
This time we need a more delicate idea than the ones highlighted in the proofs of Propositions~\ref{prop - growth} and~\ref{prop - topforest}: the {\it blow-up of an expanding real pretree rel. an expanding forest}.

\begin{thm}\label{thm - limpretrees} 
If $\phi\colon F \to F$ is an automorphism and~$\mathcal G$ a $[\phi]$-invariant proper free factor system of~$F$, then there is:
\begin{enumerate}
\item a minimal rigid $F$-action on a real pretree~${T}$ with trivial arc stabilizers;
\item a $\phi$-equivariant $F$-expanding pretree-automorphism $f\colon {T} \to {T}$; and
\item \label{cond - limpretrees.filter} an element in $F$ is ${T}$-loxodromic if and only if it grows exponentially rel.~$\mathcal G$.
\end{enumerate}
The $T$-point stabilizers give a canonical $[\phi]$-invariant subgroup system~$\mathcal H$ of finite type and any restriction~$\left.\phi\right|_{\mathcal H}$ is polynomially growing rel.~$\mathcal G$ (if~$\mathcal H$ is not empty). 
This system has strictly lower complexity than~$F$ if and only if~$\phi$ is exponentially growing rel.~$\mathcal G$.
\end{thm}

\noindent A \underline{(forward) limit pretree} of a free group automorphism is a real pretree satisfying the theorem's conclusion.

\begin{proof}[Proof (modulo a black box)]
If~$\phi$ is polynomially growing rel.~$\mathcal G$, then let~${T}$ be the degenerate real pretree (i.e.~a singleton) with a trivial $F$-action; 
Conditions~1-3 hold automatically.
For the rest of the proof, we assume that~$\phi$ is exponentially growing rel.~$\mathcal G$. 

Let~$Y^{(1)}$ be a nondegenerate tree for~$\phi$ given by Proposition~\ref{prop - topforest}.
By Gaboriau--Levitt's index inequality,~$i(F \backslash Y^{(1)}) < c(F)$, we can define~$\mathcal V$ to be a finite set of representatives for the $F$-orbits of branch points~$v$ with nontrivial stabilizers~$G_v$. 
The index inequality implies the subgroup system $\mathcal G^{(2)} = \bigsqcup_{v \in \mathcal V} G_v$ has complexity~$c(\mathcal G^{(2)}) < c(F)$ as the tree~$Y^{(1)}$ is not degenerate and the $F$-action is minimal.
The system is $[\phi]$-invariant and has a restriction automorphism~$\left.\phi\right|_{\mathcal G^{(2)}}\colon\mathcal G^{(2)} \to \mathcal G^{(2)}$ that is unique up to post-composition with inner automorphisms of $\mathcal G^{(2)}$-components --- just as we showed for free splittings in Chapter~\ref{SecPrelims}.
If~$\mathcal G^{(2)}$ is not empty and the restriction~$\left.\phi\right|_{\mathcal G^{(2)}}$ is exponentially growing rel.~$\mathcal G$, then we can repeatedly apply Proposition~\ref{prop - topforest} to the restrictions.

The complexities of the point stabilizer systems are strictly descending and, in the end, we get a finite sequence of nondegenerate forests~$\mathcal Y^{(i)}$, minimal isometric $\mathcal G^{(i)}$-actions with trivial arc stabilizers, and~$\left.\phi\right|_{\mathcal G^{(i)}}$-equivariant expanding homotheties $h^{(i)} \colon \mathcal Y^{(i)} \to \mathcal Y^{(i)}$. 
An element in~$F$ grows exponentially rel.~$\mathcal G$ if and only if it is conjugate to a  $\mathcal Y^{(i)}$-loxodromic element in some system~$\mathcal G^{(i)}$.
An almost identical argument is in the preliminaries of~\cite{Lev09}.

\medskip
Set ${T}^{(1)} \defeq {Y}^{(1)}$ with its canonical pretree structure and $f^{(1)} \defeq h^{(1)}$.
So there is a minimal rigid $F$-action on a real pretree~${T}^{(i-1)}$ with trivial arc stabilizers (for some $i\ge 2$), a $\phi$-equivariant $F$-expanding pretree-automorphism~$f^{(i-1)}$, and the ${T}^{(i-1)}$-point stabilizers are represented by~$\mathcal G^{(i)}$.

The novel construction in the paper (ideal stitching~\ref{thm - ideal stitch}, black box) defines a unique real pretree~${T}^{(i)}$ that is an equivariant blow-up of~$T^{(i-1)}$ with respect to~$f^{(i-1)}$ and~$h^{(i)}$. 
In particular, there is:

\begin{enumerate}
\item a minimal rigid $F$-action on~${T}^{(i)}$ with trivial arc stabilizers;
\item a $\phi$-equivariant $F$-expanding pretree-automorphism $f^{(i)} \colon {T}^{(i)} \to {T}^{(i)}$; and
\item $x \in F$ is ${T}^{(i)}$-loxodromic if and only if it is ${\mathcal Y}^{(j)}$-loxodromic for some $j \le i$.
\end{enumerate}
By induction, we may assume we have: a minimal rigid $F$-action on a real pretree~$T$ with trivial arc stabilizers; a $\phi$-equivariant $F$-expanding pretree-automorphism $f\colon T \to T$; and an element of~$F$ is $T$-loxodromic if and only if it is loxodromic in some forest~$\mathcal Y^{(i)}$.
By construction of the sequence of forests, the last condition translates to: an element in $F$ is ${T}$-loxodromic if and only if it grows exponentially rel.~$\mathcal G$, as required.


\medskip
Thus the $T$-point stabilizers are represented by a canonical $[\phi]$-invariant subgroup system~$\mathcal H$.
Since~$\phi$ is exponentially growing rel.~$\mathcal G$, $T$-loxodromic elements exist and~$T$ is not degenerate.
By the index inequality, we get $c(\mathcal H) < c(\mathcal F)$.
\end{proof}

We can now give the more natural characterization of polynomially growing elements:

\begin{cor}
Let $\psi\colon\mathcal F \to \mathcal F$ be an automorphism, $\mathcal G$ a $[\psi]$-invariant proper free factor system of $\mathcal F$, and $(\Lambda_*, \mathcal G)$ a free splitting of~$\mathcal F$. 
An element~$x$ in~$\mathcal F$ grows polynomially rel.~$\mathcal G$ if and only if the sequence $\left(\|\psi^n(x)\|_{\Lambda*} \right)_{n \ge 1}$ is bounded by a polynomial in~$n$.
\end{cor}

\begin{proof}
The reverse direction is immediate: polynomials are subexponential.
Suppose a nontrivial element $x \in \mathcal F$ grows polynomially rel.~$\mathcal G$.
Then it has a conjugate in~$\mathcal H$, the canonical $[\phi]$-invariant subgroup system of finite type given by Theorem~\ref{thm - limpretrees}(\ref{cond - limpretrees.filter}).
As any restriction~$\left.\phi\right|_{\mathcal H}$ is polynomially growing rel.~$\mathcal G$, we are done by Proposition~\ref{prop - growth}$(\ref{cond - growth.exprate}{\Rightarrow}\ref{cond - growth.seq})$.
\end{proof}

We also give a dendrological characterization of {\it purely exponentially growing} automorphisms. An automorphism $\phi\colon F \to F$ is \underline{atoroidal} if there are no $[\phi]$-periodic conjugacy classes of nontrivial elements in $F$. 

\begin{cor}\label{cor - atoroidal} Suppose $\phi\colon F \to F$ is a free group automorphism with a limit pretree $T$.
Then the following are equivalent:
\begin{enumerate}
\item \label{char - atoroidal} $\phi$ is atoroidal;
\item \label{char - purelyexp} all nontrivial elements in~$F$ grow exponentially; and
\item \label{char - freeaction} the $F$-action on~$T$ is free.
\end{enumerate}
\end{cor}

\begin{proof} {~}

$(\ref{char - atoroidal} \implies \ref{char - purelyexp})$: 
Set $\mathcal G \defeq \emptyset$.
If $F$ has a nontrivial element that grows polynomially, then the canonical $[\phi]$-invariant subgroup system~$\mathcal H$ of finite type given by Theorem~\ref{thm - limpretrees} is not empty and any restriction~$\left.\phi\right|_{\mathcal H}$ is polynomially growing. By Proposition~\ref{prop - growth}$(\ref{cond - growth.exprate}{\Rightarrow}\ref{cond - growth.tt})$, conjugacy classes of nontrivial elements in the lowest stratum of~$\left.\phi\right|_{\mathcal H}$ are $[\phi]$-periodic.

$(\ref{char - purelyexp} \implies \ref{char - atoroidal})$: Any nontrivial element in a $[\phi]$-periodic conjugacy class has polynomial (in fact ``constant'') growth by definition.

$(\ref{char - purelyexp} \iff \ref{char - freeaction})$: This is Theorem~\ref{thm - limpretrees}(\ref{cond - limpretrees.filter}).
\end{proof}

\section{Blow-ups: simple, na\"ive, and ideal}\label{SecBlowups}
We now define the equivariant blow-ups that were the key steps in the proofs of Proposition~\ref{prop - topforest} and Theorem~\ref{thm - limpretrees}.
Although we present the constructions as two separate ideas (simplicial and non-simplicial), the underlying principle is the same:

\begin{enumerate}
\item start with an $F$-action on a tree (or real pretree) $T$ and a $\mathcal G$-action on a forest $\mathcal T_{\mathcal V}$ where~$\mathcal G$ are conjugacy representatives of the nontrivial point stabilizers of $T$;
\item then define a {\it na\"ive} equivariant blow-up of $T$ with respect to $\mathcal T_{\mathcal V}$;
\item there is a lot of freedom in the na\"ive construction but most choices will not be useful;
\item let $f\colon T \to T$ and $f_{\mathcal V}\colon\mathcal T_{\mathcal V} \to \mathcal T_{\mathcal V}$ be $\phi$- and $\left.\phi\right|_{\mathcal G}$-equivariant homeomorphisms respectively;
\item  consider the equivariant copies of $\mathcal T_{\mathcal V}$ in some blow-up tree (or real pretree)~$T^*$, then $f$ induces a $\phi$-equivariant function $f^*\colon T^* \to T^*$ whose restriction to the copies is $f_{\mathcal V}$ --- in a way, $f^*$ is formed by {\it stitching} $f$ and $f_{\mathcal V}$ together;
\item for most blow-ups, $f^*$ will not be a homeomorphism; but if $f_{\mathcal V}$ is {\it expanding}, then a unique fixed point theorem produces the ``ideal'' blow-up $T^*$ whose corresponding map $f^*$ is a homeomorphism.
\end{enumerate}

\subsection{Simple patchwork}
In this case, a blow-up is easy-but-tedious to define; ensuring the induced map is a homothety is the tricky bit.

\begin{thm}\label{thm - simple patch} Let $\psi\colon \mathcal F \to \mathcal F$ be an automorphism of a free group system $\mathcal F$ and $(\Gamma_*, \mathcal G)$ a free splitting of $\mathcal F$. Assume there is:
\begin{enumerate}
\item  a $\psi$-equivariant simplicial automorphism $\tau \colon (\Gamma_*, \mathcal G) \to (\Gamma_*, \mathcal G)$,
\item a minimal isometric $\mathcal G$-action on a forest $\mathcal T_{\mathcal V}$ with trivial arc stabilizers; and
\item a $\left.\psi\right|_{\mathcal G}$-equivariant expanding $\lambda$-homothety $h_{\mathcal V}\colon \mathcal T_{\mathcal V} \to \mathcal T_{\mathcal V}$.
\end{enumerate}
Then there is:
\begin{enumerate}
\item a minimal isometric $\mathcal F$-action on a forest $\mathcal T$ with trivial arc stabilizers, an equivariant copy of~$\mathcal T_{\mathcal V}$, and
\item a $\psi$-equivariant expanding $\lambda$-homothety $h\colon \mathcal T \to \mathcal T$ induced by~$\tau$ and~$h_{\mathcal V}$.
\end{enumerate}
In fact, the $\mathcal F$-action on $\mathcal T$ decomposes as a \emph{graph of actions} whose underlying simplicial action is the free splitting $(\Gamma_*, \mathcal G)$ and the vertex actions are the given $\mathcal G$-action on $\mathcal T_{\mathcal V}$.
The homothety~$h$ acts by~$\tau$ on the underlying simplicial action and $h_{\mathcal V}$ on the vertex actions.
Any pair $(\mathcal T', h')$ satisfying this conclusion admits an equivariant isometry $\mathcal T' \to \mathcal T$ that conjugates~$h'$ to~$h$.
\end{thm}

\noindent We state and prove the theorem in terms of homotheties due to the specific needs in Proposition~\ref{prop - topforest} but the argument actually holds if ``expanding homothety'' is replaced with ``expansions (or contractions)''. 
We only need a hypothesis that lets us use the contraction mapping theorem.

\begin{proof} Suppose $(\Gamma_*, \mathcal G)$ is a free splitting of $\mathcal F$. 
Let $\mathcal V$ be a set of orbit representatives for vertices in the free splitting with nontrivial stabilizers. 
For each $v \in \mathcal V$, let $\mathcal D_v$ be a set of orbit representatives for half-edges originating from $v$ and~$\overline T_v$ the metric completion of the $\mathcal T_{\mathcal V}$-component corresponding to~$G_v$.

Suppose $(p_d \in \overline T_v : v \in \mathcal V, d \in \mathcal D_v)$ is a choice of {\it attaching points} in~$\overline T_{\mathcal V}$.

\medskip
\noindent ({\it graph of actions}) 
To simplify the discussion, we will pretend $\mathcal F=F$ is connected for the moment.
Let $E$ be the complement of the $F$-orbit of $\mathcal V$ in $(\Gamma, \mathcal G)$;
essentially, we just want the half-edges in~$\mathcal D_v$ to now have distinct origins.
$E$ inherits a free $F$-action from the free splitting.
Set $\mathcal V^* \defeq F \times \overline{\mathcal T}_{\mathcal V}$;
$F$ acts on $\mathcal V^*$ by left-multiplication on the first factor.
The equivariant blow-up $T^* = T^*(p_d : v \in \mathcal V, d \in \mathcal D_v)$ of $(\Gamma, \mathcal G)$ with respect to the forest $\mathcal T_{\mathcal V}$ is defined through the quotient $\mathcal V^* \sqcup E \overset{\iota}\to T^*$ given by the following identifications: 
\begin{enumerate}
\item for each $v \in \mathcal V$, $d \in \mathcal D_v$ and $x \in F$, the origin of the half-edge $x \cdot d$ in~$E$ is identified with $(x, p_d)$ in~$\mathcal V^*$, i.e.~the half-edge $d$ is {\it attached} to $p_d$ equivariantly; and
\item for each $v \in \mathcal V$, $p \in \overline T_v$, $s \in G_v$, and $x\in F$, the point $(xs, p)$ in~$\mathcal V^*$ is identified with $(x, s \cdot p)$ in~$\mathcal V^*$.
\end{enumerate}
As the $F$-actions on $\mathcal V^*$ and $E$ respect the identifications, the blow-up~$T^*$ inherits an $F$-action by homeomorphisms of the quotient topology.
For each $s \in G_v$, the image~$\iota( s, \overline T_v)$ is a $G_v$-equivariant copy of $\overline T_v$ since $(\Gamma, \mathcal G)$ has trivial edge stabilizers; 
in general, we would need each attaching point~$p_d \in \overline T_v$ to be fixed by the stabilizer of $d$ in~$G_v$.

\medskip
The blow-up $T^*$ with the quotient topology has a natural projection $\pi \colon T^* \to (\Gamma, \mathcal G)$ whose point-preimages are connected: single point or a copy of $\overline T_v$. 
This implies $T^*$ is {\it uniquely arcwise connected}. 
Since the vertex set of $(\Gamma, \mathcal G)$ is a discrete subspace with finitely many $F$-orbits, any closed arc in $T^*$ decomposes into finitely many subarcs that are in $\iota(E)$ or $\iota(\mathcal V^*)$. 
Thus, closed arcs in $T^*$ inherit lengths from $(\Gamma_*, \mathcal G)$ and~$\overline{\mathcal T}_{\mathcal V}$; 
this path metric makes~$T^*$ a (metric) tree. 
By the equivariant construction, the tree~$T^*$ inherits an isometric $F$-action with trivial arc stabilizers and an equivariant copy of~$\mathcal T_{\mathcal V}$. 

The $F$-action on~$T^*$ is what Levitt calls a {\it graph of actions}~\cite[p.~32]{Lev94} --- the underlying simplicial action is the free splitting $(\Gamma, \mathcal G)$. 
Different choices $(p_d : v \in \mathcal V, d \in \mathcal D_v)$ of attaching points may produce drastically different graphs of actions.

\medskip
Returning to the general case where~$\mathcal F$ is possibly disconnected, we can apply the blow-up construction componentwise to get a forest~$\mathcal T^*$ with an isometric $\mathcal F$-action.
We have yet to use the $\psi$-equivariant simplicial automorphism $\tau\colon (\Gamma_*,\mathcal G) \to (\Gamma_*, \mathcal G)$ nor the $\left.\psi\right|_{\mathcal G}$-equivariant expanding $\lambda$-homothety $h_{\mathcal V} \colon \mathcal T_{\mathcal V} \to \mathcal T_{\mathcal V}$.

\begin{rmk} We give two related proofs of the conclusion to the theorem. 
The first proof is short but it does not generalize to non-metric settings --- we only sketch it. 
The second proof is thorough as it contains the ideas needed later for the main construction.
\end{rmk}

\medskip
\noindent ({\it projective limit}) Suppose we have a na\"ive blow-up~$\mathcal T^*$ (made from an arbitrary choice of attaching points) and let~$\| \cdot \|^*$ be its translation distance function.
Since $\tau$ was a simplicial automorphism and $h_{\mathcal V}$ is a $\left.\psi\right|_{\mathcal G}$-equivariant $\lambda$-homothety, we get $\lambda^{-n} \|\psi^n(\cdot)\|^* \to \| \cdot \|$, where the limit function~$\| \cdot \|$ is the translation distance function for the required $\mathcal F$-action.

\medskip
\noindent For the second proof, we get the ideal attaching points without taking projective limits.

\medskip

\noindent ({\it ideal stitching})
As the simplicial automorphism $\tau$ permutes the (finitely many) orbits of vertices and half-edges, it induces permutations $\beta \in \mathrm{Sym}(\mathcal V)$ and $\partial \in \mathrm{Sym}(\bigcup_{v\in \mathcal V} \mathcal D_v)$.
The maps~$\tau$ and $h_{\mathcal V}$ induce a $\psi$-equivariant {\it PL-map} $f^*\colon {\mathcal T^*} \to {\mathcal T^*}$: 
\begin{itemize}
\item let $\nu\colon \mathcal V^* \to \mathcal \iota(\mathcal V^*)$ be given by $ \nu(x,p) \defeq \iota(\psi(x)x_v, \bar h_v(p))$, where $( x_v \in \mathcal F: v\in \mathcal V)$ was the implicit choice used to define the restriction $\left.\psi\right|_{\mathcal G}$, and hence~$h_{\mathcal V}$ and its extension $\bar h_{\mathcal V} \colon \overline{\mathcal T}_{\mathcal V} \to \overline{\mathcal T}_{\mathcal V}$. 
Observe that $\nu(xs, p) = \nu(x, s\cdot p)$ when $s \in G_v$ (exercise). 

let $\epsilon\colon E \to \mathcal T^*$ be given by:
\item for each half-edge $e$ in $E$ not in the $\mathcal F $-orbit of $\bigcup_{v\in \mathcal V} \mathcal D_v$, $\epsilon$ maps $e$ onto $\iota(\tau(e)) \subset E$ along an orientation-preserving isometry.
\item for each $v \in \mathcal V$, $d \in \mathcal D_v$, and element $x$ in the $\mathcal F$-component that contains $G_v$, $\epsilon$~maps $x \cdot d$ onto the concatenation of 
$\iota( \psi(x)x_v , [\bar h_v(p_d), s_{\partial \cdot d} \cdot p_{\partial \cdot v}]_{\partial \cdot v} )$
and $\iota(\psi(x) \cdot \tau(d))$ along an orientation-preserving linear map, 
where $ \tau(d) = x_v s_{\partial \cdot d} \cdot \partial(d)$ --- $s_{\partial \cdot d} \in G_{\beta \cdot v}$ is unique as edge stabilizers are trivial.
Notice that the origin of $\epsilon( x \cdot d)$ is $\nu(x,p_d)$.
\end{itemize}
By definition, $\nu \sqcup \epsilon\colon \mathcal V^* \sqcup E \to \mathcal T^*$ factors through $\iota$ to induce a PL-map~$f^*$.

\medskip
The PL-map $f^*$ is injective if and only if $\bar h_v(p_d) = s_{\partial \cdot d} \cdot p_{\partial \cdot d}$ for all $v \in \mathcal V$, $d \in \mathcal D_v$.
This finite system of equations with unknowns $( p_d \in \overline T_v : v \in \mathcal V, d \in \mathcal D_v )$ is restated as: 
\[ p_d = h_{v,d}(p_{\partial \cdot d}) ~\text{ for all } v\in \mathcal V,  d \in \mathcal D_v, \quad \text{where } h_{v,d} \defeq {\bar h}_v^{\,-1} \circ s_{\partial \cdot d} \colon T_{\beta \cdot v} \to T_v.\]
Note that minimality of~$\mathcal T_{\mathcal V}$ implies $h_{\mathcal V}$ is surjective and hence invertible.
So $h_{v,d}\colon T_{\beta \cdot v} \to T_v$, the composition of an isometry with a {\it contracting} homothety, is a contracting homothety. 
As $\bigcup_{v\in \mathcal V} \mathcal D_v$ is finite and~$\overline{\mathcal T}_{\mathcal V}$ is complete, our system of equations has a unique solution by the contraction mapping theorem, i.e.~there is a unique tuple $\vec p = (p_d \in \overline T_v : v \in \mathcal V, d \in \mathcal D_v)$ with  $p_d = h_{v,d}(p_{\partial \cdot d})$ for all $v \in \mathcal V$, $d \in \mathcal D_v$.


\medskip
Let $\mathcal T^* = \mathcal T^*(\vec p\,)$ be the blow-up given by this unique solution.
Then the homeomorphism~$f^*$ is isometric on permuted components of the simplicial part $\iota(E)$ and an expanding $\lambda$-homothety on components of the non-simplicial part $\iota(\mathcal V^*)$. 

\medskip
{\noindent \it (simplicial collapse)} To finish the proof, we collapse the simplicial edges to get an isometric $\mathcal F$-action on a forest with trivial arc stabilizers. 
Let $\mathcal T$ be the characteristic subforest for $\mathcal F$, i.e.~the $\mathcal F$-action on $\mathcal T$ is minimal. 
Since we only collapsed the simplicial part, the PL-map~$f^*$ induces a $\psi$-equivariant expanding $\lambda$-homothety $h\colon \mathcal T \to \mathcal T$. 

\medskip
{\noindent \it (uniqueness)} Suppose some $\mathcal F$-forest $\mathcal T'$ decomposed as a graph of actions with underlying simplicial action $(\Gamma_*, \mathcal G)$ and vertex actions~$\mathcal T_{\mathcal V}$ and admitted a homothety~$h'$ that acts by~$\tau$ on $(\Gamma_*, \mathcal G)$ and $h_{\mathcal V}$ on~$\mathcal T_{\mathcal V}$.
This forest must arise from the above construction and the equivariant isometric identification with $\mathcal T$ follows from uniqueness of the solution~$\vec p$.
\end{proof}

\subsection{Na\"ive stitching}
Due to the involved nature of the proofs when dealing with non-simplicial trees, the blow-up and stitching constructions are split into two sections;
let us start with blow-ups.

Suppose a countable group~$G$ acts rigidly on a real pretree~$T$.
Recall that if the action is minimal, then~$T$ must be short and the pretree completion~$\widehat T$ is itself real.
Given any subset $P \subset \widehat T$, we define~$T[P]$ to be the union in~$\widehat T$ of $T$ and the $G$-orbit of~$P$.
Generally, assume a countable group system $\mathcal G = \bigsqcup_{i \in \mathcal I} G_i$ acts minimally on real pretrees~$\mathcal T_{\mathcal I}$.
Then for any subset $P \subset \widehat{\mathcal T}$, we can similarly define~$\mathcal T[P]$ and let $T_i[P]$ be the component of~$\mathcal T[P]$ indexed by $i \in \mathcal I$.
We now make a simple but crucial observation:

\begin{lem}\label{lem - rigid extension}Let~$\mathcal G$ be a countable group system that admits a minimal rigid action on real pretrees~$\mathcal T$.
If a subset $P \subset \widehat{\mathcal T}$ contains no points fixed by a $\mathcal T$-loxodromic element of~$\mathcal G$, then the induced $\mathcal G$-action on the real pretrees~$\mathcal T[P]$ is rigid.

Additionally, if the rigid action on~$\mathcal T$ has finite arc stabilizers, then so does the rigid action on~$\mathcal T[P]$.
\end{lem}
\begin{proof}[Sketch of proof]
$\mathcal T[P]$-loxodromics are precisely $\mathcal T$-loxodromics as no point in~$P$ is fixed by a $\mathcal T$-loxodromic element of~$\mathcal G$.
Also, observe that, for all elements~$g$ in $\mathcal G$, the $\mathcal T[P]$-directions at $\mathrm{Fix}_{\mathcal T[P]}(g)$ are in one-to-one equivariant bijection with the $\mathcal T$-directions at $\mathrm{Fix}_{\mathcal T}(g)$.
So the induced action on~$\mathcal T[P]$ is rigid.
Finally, any arc of~$\mathcal T[P]$ contains an arc of~$\mathcal T$; 
if the latter have finite stabilizers,
then the rigid action on~$\mathcal T[P]$ has finite arc stabilizers.
\end{proof}

For a given countable group~$H$ and rigid $H$-action on a real pretree~$T$ with finite arc stabilizers, fix a set of orbit representatives $\mathcal V \subset T$ for branch points with infinite stabilizers~$\mathcal G$.
By countability of~$H$ and finiteness of arc stabilizers, the set~$\mathcal V$ is countable.
For each $v\in \mathcal V$, consider the $G_v$-action on the directions at~$v$ and fix a set of orbit representatives~$\mathcal D_v$ for these directions.
For a direction~$d$ at~$v$, we let $\mathrm{Stab}_{G_v}(d)$ denote the (setwise) stabilizer of~$d$ in~$G_v$. 

\begin{prop}\label{prop - naive stitching} 
Suppose a real pretree~${T}$ has a rigid $H$-action with finite arc stabilizers, $\mathcal G$ represents the infinite ${T}$-point stabilizers, and there is a minimal rigid $\mathcal G$-action on real pretrees~${\mathcal T}_{\mathcal V}$.

If $\vec p = ( p_d \in \widehat T_v: v \in \mathcal V, d \in \mathcal D_v)$ is a choice of points  and each point~$p_d$ is fixed by the direction stabilizer $\mathrm{Stab}_{G_v}(d)$, then there is:
\begin{enumerate}
\item \label{cond - naive.pretree} an $H$-action on a real pretree~$T^* = T^*(\vec p\,)$ by pretree-automorphisms; 
\item \label{cond - naive.pretree of actions} equivariantly collapsing the copies of~${\mathcal T}_{\mathcal V}[\,\vec p \,]$ in~$T^*$ recovers the real pretree~${T}$;
\end{enumerate}
If no attaching point~$p_d$ in~$\vec p$ is fixed by a ${\mathcal T}_{\mathcal V}$-loxodromic element in~$\mathcal G$, then
\begin{enumerate}[resume]
\item \label{cond - naive.rigid} the $H$-action on~$T^*$ is rigid;
\item \label{cond - naive.charrealpretree} the characteristic convex subset of~$T^*$ for~$\mathcal G$ is (an equivariant copy of)~${\mathcal T}_{\mathcal V}$; and
\item \label{cond - naive.loxodromics} an element in~$H$ is $T^*$-loxodromic if and only if it is ${T}$- or ${\mathcal T}_{\mathcal V}$-loxodromic.
\end{enumerate}
Moreover, if the $\mathcal G$-action on~$\mathcal T_{\mathcal V}$ has finite arc stabilizers, then so does the $H$-action on~$T^*$. 
If the rigid $H$-action on~$T$ is minimal, then so is the rigid $H$-action on~$T^*$.
\end{prop}

\noindent 
For the sake of compartmentalizing the proposition's long proof, we break it into two parts. 


\begin{proof}[Proof of real pretree's existence] 
Following the discussion preceding the theorem's statement, let~$T_v~(v \in \mathcal V)$ be the component of~${\mathcal T}_{\mathcal V}$ corresponding to~$G_v$.

Assume $\vec p = ( p_d \in \widehat T_v: v \in \mathcal V, d \in \mathcal D_v)$ is a choice of points where each point~$p_d$ is fixed by the direction stabilizer $\mathrm{Stab}_{G_v}(d)$;
these points will be referred to as the {\it attaching points}.
The real pretrees~$T$ and~$T_v[\,\vec p\,]~(\text{for } v \in \mathcal V)$ have pretree structures denoted 
$[\cdot, \cdot]$ and $[\cdot , \cdot]_v$ respectively.
The equivariant blow-up~$T^* = T^*(\vec p\,)$ of~$T$ with respect to ${\mathcal T}_{\mathcal V}[\,\vec p\,]$ is defined through the quotient
\[ \iota \colon T \setminus (H \cdot \mathcal V)~\sqcup~ H \times {\mathcal T}_{\mathcal V}[\,\vec p\,] \longrightarrow T^*,\]
given by these identifications:
\begin{itemize}
\item for each $v \in \mathcal V$, $p \in T_v[\,\vec p\,]$, $s \in G_v$, and $x \in H$, identify~$(xs, p)$ with $(x, s \cdot p)$.
\end{itemize}
Just as in the simplicial setting, $T^*$ inherits an $H$-action --- at least by set-bijections.
By construction,~$T_v[\,\vec p \,]$ is equivariantly identified with $\iota(G_v \times T_v[\,\vec p \,])$ for $v \in \mathcal V$ and equivariantly collapsing the translates of~$\mathcal T_{\mathcal V}[\,\vec p\,]$ produces an equivariant surjection $\pi \colon T^* \to T$. 
This almost establishes Condition~\ref{cond - naive.pretree of actions}:
we need to define a pretree structure that is $H$-invariant and projects to the pretree structure of~$T$ under~$\pi$.

\medskip
For a pretree structure, we need a function $[\cdot, \cdot]^*\colon T^* \times T^* \to \mathcal P(T^*)$. Let $p^*, q^* \in T^*$.

\medskip
\noindent \underline{Case 1: $\pi(p^*) = \pi(q^*) \in H \cdot \mathcal V $}. 
Choose $y \in H$ so that $\pi(p^*) = y \cdot v$ with $v \in \mathcal V$. Then $p^* = \iota(y, p)$ and $q^* = \iota(y, q)$ for some $p,q \in T_v[\,\vec p\,]$. Define 
\[ [p^*, q^*]^* \defeq \iota( y, [p,q]_v). \]
Let us quickly check that $[p^*, q^*]^*$ is independent of our choice of $y \in H$. If $s \in G_v$, then $\pi(p^*) = ys \cdot v$, $p^* = \iota(ys, p')$, and $q^* = \iota(ys, q')$ where $s\cdot p' = p$ and $s\cdot q' = q$ by $\iota$'s definition. 
So 
\( \iota (ys , [p', q']_v) = \iota ( y, s \cdot [p', q']_v) = \iota ( y, [p,q]_v). \)

In this case, we are equivariantly extending the pretree-structure of~$\mathcal T_{\mathcal V}[\,\vec p\,]$ to its orbit in~$T^*$, i.e.~the image~$\iota(H \times \mathcal T_{\mathcal V}[\,\vec p\,])$.

\medskip
\noindent \underline{Case 2: both $\pi(p^*) = p, \pi(q^*) = q$ are not in $H \cdot \mathcal V$}. 
Consider the closed interval $[p,q] \subset T$. Then 
\[ [p,q] = [p,q] \setminus (H \cdot \mathcal V)~\cup \bigcup_{v \in \mathcal V, k \ge 1} \{ y_{v,k} \cdot v \}, \]
for some (possibly finite or empty) subset $\{ y_{v,k} \}_{v \in \mathcal V, k \ge 1}$ of~$H$. 
For each $v \in \mathcal V$ and $k \ge 1$, the interval $[p, y_{v,k} \cdot v)$ is contained in the direction $y_{v,k} s_{v, k_{-}} \cdot d(v,k_{-})$ at $y_{v,k} \cdot v$ for some $s_{v, k_{-}} \in G_v$ and $d(v,k_{-}) \in \mathcal D_v$. 
The element $s_{v, k_{-}}$ is unique up to right-multiplication by~$\mathrm{Stab}_{G_v}(d(v,k_{-}))$.
So the point $s_{v, k_{-}}\cdot p_{d(v,k_{-})}$ is independent of the choice of~$s_{v,k_{-}}$ since the attaching point~$p_{d(v,k_{-})}$ is fixed by~$\mathrm{Stab}_{G_v}(d(v,k_{-}))$.
Similarly, $(y_{v,k} \cdot v, q]$ is in $y_{v,k} s_{v, k_{+}} \cdot d(v, k_{+})$ at $y_{v,k} \cdot v$.

Define \[ [p^*,q^*]^* \defeq \iota{\left( [p,q] \setminus (H \cdot \mathcal V) ~\cup \bigcup_{v \in \mathcal V, k \ge 1} \{ y_{v,k} \} \times [s_{v,k_{-}} \cdot p_{d(v,k_{-})},\,s_{v,k_{+}} \cdot p_{d(v,k_{+})}]_v \right)}. \]
Again, $[p^*, q^*]^*$ is independent of the choice of $y_{v,k}$. 
In this case (and the rest), we are equivariantly ``attaching'' directions~$d$ to points $p_d$ to get the pretree~$T^*(\vec p\,)$. 

\medskip
\noindent \underline{Case 3: $\pi(p^*) = p \notin H \cdot \mathcal V, \pi(q^*) \in H \cdot \mathcal V $ (or vice-versa)}. 
Choose an element $y_\omega \in H$ so that $\pi(q^*) = y_\omega \cdot w$ with $w \in \mathcal V$. Then $q^* = \iota(y_\omega, q)$ for some $q \in  T_w$.
Consider the closed interval $[p,y_\omega \cdot w] \subset  T$. Then 
\[ [p, y_\omega \cdot w] = [p, y_\omega \cdot w] \setminus (H \cdot \mathcal V)~\cup \{ y_\omega \cdot w \} \cup \bigcup_{\substack{v \in \mathcal V, k \ge 1 \\ y_{v,k} \cdot v \neq y_\omega \cdot w}} \{ y_{v,k} \cdot v \}. \]

As in the previous case, for each $v \in \mathcal V$ and $k \ge 1$, the interval $[p, y_\omega \cdot w]$ determines directions $y_{v,k} s_{v,k_\pm} \cdot d(v,k_\pm)$ at the interior point $y_{v,k} \cdot v$.
Additionally, the interval determines some direction $y_\omega s_\omega \cdot d(\omega)$ at $y_\omega \cdot w$.

Define $[p^*,q^*]^*$ to be the $\iota$-image of \[ [p,y_\omega \cdot w] \setminus (H \cdot \mathcal V) ~\cup \left(\{ y_\omega \} \times [s_\omega \cdot p_{d(\omega)},\,q]_w \right) \cup \bigcup_{\substack{v \in \mathcal V, k \ge 1 \\ y_{v,k} \cdot v \neq y_\omega \cdot w}} \{ y_{v,k} \} \times [s_{v,k_{-}} \cdot p_{d(v,k_{-})},\,s_{v,k_{+}} \cdot p_{d(v,k_{+})}]_v. \]

\noindent \underline{Case 4: $\pi(p^*), \pi(q^*) \in H \cdot \mathcal V$ with $\pi(p^*) \neq \pi(q^*)$}.
Choose $y_\alpha, y_\omega \in H$ so that $\pi(p^*) = y_\alpha \cdot a$ and $\pi(q^*) = y_\omega \cdot w$ with $a, w \in \mathcal V$.
The rest is just like Case~3 except that this time $[y_\alpha \cdot a, y_\omega \cdot w]$ determines directions $y_\alpha s_\alpha \cdot d(\alpha)$ and $y_\omega s_\omega \cdot d(\omega)$ at $y_\alpha \cdot a$ and $y_\omega \cdot w$ respectively. 
We omit the redundant details.
\medskip

By construction, $\pi([p^*,q^*]^*)=[\pi(p^*), \pi(q^*)]$. 
From this, we can directly verify that: 
\begin{enumerate}
\item $[\cdot, \cdot]^*$ satisfies the pretree axioms (see Appendix~\ref{SecSolns} for \hyperlink{prop - naive stitching}{details}). 

\item $H$ acts on $T^*$ by pretree-automorphisms, i.e.~for all $x \in H$ and $p^*,q^* \in T^*$, we have $[x \cdot p^*, x \cdot q^*]^* = x \cdot [p^*,q^*]^*$ (exercise).
\end{enumerate}

\medskip
Remember that $(T, [\cdot, \cdot])$ and $(T_v[\,\vec p\,], [\cdot, \cdot]_v)~(v \in \mathcal V)$ were real pretrees.
Let us now check that the pretree $(T^*, [\cdot, \cdot]^*)$ is real, i.e.~any closed interval $I^* = [p^*, q^*]^*$ is pretree-isomorphic to a closed interval of~$\mathbb R$. 
By the definition of $[\cdot, \cdot]^*$, $I^*$ is formed by removing countably many points from a closed interval $I \subset T$ and replacing them with closed intervals pretree-isomorphic to closed intervals of ${\mathcal T}_{\mathcal V}[\,\vec p\,]$-components. Note that~$I$ is pretree-isomorphic to a closed interval of $\mathbb R$ since $T$ is a real pretree. 
For the same reason, closed intervals of ${\mathcal T}_{\mathcal V}[\,\vec p\,]$-components are pretree-isomorphic to closed intervals of $\mathbb R$. 
Similar to the Cantor function, we can use the given pretree-isomorphisms (and the axiom of choice) to construct an isomorphism between $I^*$ and a closed interval of $\mathbb R$. We leave the details as an exercise. So~$T^*$ is a real pretree.
\end{proof}

\begin{rmk}
Although we invoked the axiom of (countable) choice to conclude that the blow-up~$T^*$ is real, this can be avoided in the setting we are most intersted in: $H = F$ and the action is minimal.
In this case, Gaboriau--Lustig's index theory implies there are only finitely many orbits of branch points and hence only finitely many choices to be made.
\end{rmk}

\begin{proof}[Proof that conditions are satisfied...] 
Now assume none of the attaching points~$\vec p$ are fixed by a ${\mathcal T}_{\mathcal V}$-loxodromic element in~$\mathcal G$.

\medskip 
\textbullet ~Condition~\ref{cond - naive.rigid}, rigidity of $H$-action on~$T^*$:

\noindent 
Suppose $x \in H$ has a nonempty fixed-point set $\mathrm{Fix}_{T^*}(x) \subset T^*$. 
We need to show that~$x$ fixes no $T^*$-direction at $\mathrm{Fix}_{T^*}(x)$.

First, we show that the fixed-point set $\mathrm{Fix}_{T^*}(x)$ is convex.
By the equivariance of~$\pi$ and rigidity of the $H$-action on~$T$, $\mathrm{Fix}_T(x)$ is nonempty and convex.
As~$x$ acts as a pretree-automorphism, it must also fix the limit points of $\iota(\mathrm{Fix}_{T}(x) \setminus \left( H\cdot \mathcal V\right))$ in $\iota(H \times \mathcal T_{\mathcal V}[\,\vec p\,])$.
By Lemma~\ref{lem - rigid extension}, the $\mathcal G$-action on~$\mathcal T_{\mathcal V}[\,\vec p\,]$ is rigid and so $\mathrm{Fix}_{T^*}(x)$ is convex:
this used the assumption~$\vec p$ was a choice of points not fixed by $\mathcal T_{\mathcal V}$-loxodromic elements in~$\mathcal G$.
Thus each $T^*$-direction~$d$ at $\mathrm{Fix}_{T^*}(x)$ has a unique attaching point~$p^*$ in $\mathrm{Fix}_{T^*}(x)$.

Finally, pick a $T^*$-direction~$d$ at $\mathrm{Fix}_{T^*}(x)$ and let~$p^* \in \mathrm{Fix}_{T^*}(x)$ be its attaching point. 
If $\pi(p^*) \notin H\cdot \mathcal V$, then~$d$ corresponds to  a $T$-direction at $\mathrm{Fix}_T(x)$ and the latter is not fixed by~$x$ as the $H$-action on~$T$ is rigid.
If $\pi(p^*) = y \cdot v$ for some $y \in H$ and $v \in \mathcal V$, then~$d$ corresponds to either a $\iota(y, T_{v})$-direction at $\iota(y, T_{v}) \cap \mathrm{Fix}_{T^*}(x)$ or a $T$-direction at $\mathrm{Fix}_T(x)$ (possible when $p^* \in \iota(y, \vec p\,)$).
Again, in either case, the direction is not fixed due to rigidity of actions.
As~$d$ was arbitrary, we are done:~$x$ fixes no $T^*$-direction at $\mathrm{Fix}_{T^*}(x)$.

\medskip
\textbullet ~Condition~\ref{cond - naive.charrealpretree}, characteristic convex subset of~$T^*$ for~$\mathcal G$: 

\noindent Fix a point $v \in \mathcal V$.
We need to show that the characteristic convex subset of~$T^*$ for~$G_v$ is~$\iota(G_v \times T_v)$.
Since the rigid $H$-action on~$T$ has finite arc stabilizers and the $T$-elliptic subgroup~$G_v$ is infinite, the characteristic convex subset of~$T$ for~$G_v$ is the singleton~$\{ v \}$. 
Thus, by the equivariance of $\pi$, a characteristic convex subset of~$T^*$ for~$G_v$ is contained in $\pi^{-1}(v) = \iota(G_v \times T_v[\,\vec p\,])$.
But~$\pi^{-1}(v)$ and~${T}_v[\,\vec p\,]$ are equivariantly pretree-isomorphic (Condition~\ref{cond - naive.pretree of actions}).
The characteristic convex subset of~$T_v[\,\vec p\,]$ for~$G_v$ is~$T_v$ since we assumed the $G_v$-action on~$T_v$ is minimal;
therefore, the characteristic convex subset of~$T^*$ for~$G_v$ is~$\iota(G_v \times T_v)$ as needed.

\medskip
\textbullet ~Condition~\ref{cond - naive.loxodromics}, characterizing $T^*$-loxodromics --- or equivalently, $T^*$-elliptics:

\noindent Suppose $x$ is $T$- and ${\mathcal T}_{\mathcal V}$-elliptic, i.e.~after replacing it with a conjugate if necessary, $x \in G_v$ for some $v \in \mathcal V$ and is also $T_v$-elliptic.
So $\iota(x, \mathrm{Fix}_{T_v}(x)) \subset \mathrm{Fix}_{T^*}(x)$ and $x$ is $T^*$-elliptic.

Conversely, suppose $x$ is $T$- or $\mathcal T_{\mathcal V}$-loxodromic.
By the equivariance of~$\pi$, $\pi(\mathrm{Fix}_{T^*}(x))$ is contained in $\mathrm{Fix}_T(x)$.
If $x$ is $T$-loxodromic, i.e.~$\mathrm{Fix}_T(x) = \emptyset$, then it is also $T^*$-loxodromic.
Suppose $x$ is $T$-elliptic but $\mathcal T_{\mathcal V}$-loxodromic.
Then $\mathrm{Fix}_{T}(x)$ is a singleton since the rigid $H$-action on~$T$ has finite arc stabilizers.
After replacing it with a conjugate if necessary, assume $x \in G_v$ and is $T_v$-loxodromic for some $v\in \mathcal V$.
As $\mathrm{Fix}_{T^*}(x)$ is contained in $\pi^{-1}(v)$ and the pretrees~$\pi^{-1}(v)$ and~$T_v[\,\vec p\,]$ are equivariantly identified, we get $\mathrm{Fix}_{T^*}(x)$ is $\iota(x, \mathrm{Fix}_{T_v[\,\vec p\,]}(x))$.
But~$x$ is $T_v[\,\vec p\,]$-loxodromic since it is $T_v$-loxodromic and it fixes none of the attaching points~$\vec p$ by assumption;
therefore,~$x$ is $T^*$-loxodromic.

\medskip
\textbullet ~Finiteness of arc stabilizers:

\noindent Assume the rigid $\mathcal G$-action on $\mathcal T_{\mathcal V}$ has finite arc stabilizers.
Let $A^* \subset T^*$ be an arc, then $A \defeq \pi(A^*) \subset T$ is either a singleton or an arc.
If $A$ is an arc, then, by the equivariance of~$\pi$ and rigidity of actions (Condition~\ref{cond - naive.rigid}), the stabilizer of~$A^*$ is the stabilizer of~$A$;
so it is finite.
If $A = \{ p \}$ is a singleton, then $p = y \cdot v$ for some $y\in H$ and $v \in \mathcal V$.
Thus $A^* = \iota(y, A_v)$ for some arc $A_v \subset T_v[\,\vec p\,]$ and its stabilizer in~$H$ is conjugate to the stabilizer of~$A_v$ in~$G_v$;
therefore, it is finite by Lemma~\ref{lem - rigid extension}.

\medskip
\textbullet ~Minimality:

\noindent Assume the rigid $H$-action on~$T$ is minimal.
By the equivariance of~$\pi$, any characteristic convex subset of~$T^*$ for~$H$ must map onto~$T$ under the projection~$\pi$.
In particular, any characteristic convex subset will contain~$\iota(T \setminus (H \cdot \mathcal V))$ and the attaching points $\iota(H \times \vec p\,)$.
By Condition~\ref{cond - naive.charrealpretree}, any characteristic convex subset will contain~$\iota(H \times \mathcal T_{\mathcal V})$;
therefore,~$T^*$ is the characteristic convex subset for~$H$, i.e.~the $H$-action on $T^*$ is minimal.
\end{proof}

\subsection{Ideal stitching}\label{subsec - blowups.ideal}

The crucial step in proving Theorem~\ref{thm - simple patch} was using the contraction mapping theorem to find the right blow-up that induced a homothety.
We use the same idea for the next theorem, which is the heart of the paper.

For a minimal rigid $F$-action on a real pretree with trivial arc stabilizers, let~$\mathcal V$ be a set of orbit representatives of branch points with nontrivial stabilizers~$G_v$. The point stabilizer system $\mathcal G = \bigsqcup_{v \in \mathcal V} G_v$ has finite type by Gaboriau--Levitt's index inequality.

\begin{thm} \label{thm - ideal stitch} Suppose a real pretree~$T$ has a minimal rigid $F$-action with trivial arc stabilizers, and let $\phi\colon F \to F$ be an automorphism.

Assume $f \colon T \to T$ is a $\phi$-equivariant pretree-automorphism and~$\mathcal G$ represents the nontrivial $T$-point stabilizers with:
\begin{enumerate}
\item a minimal isometric $\mathcal G$-action on a forest~${\mathcal Y}_{\mathcal V}$ with trivial arc stabilizers; and
\item a $\left.\phi\right|_{\mathcal G}$-equivariant expanding homothety $h_{\mathcal V}\colon {\mathcal Y}_{\mathcal V} \to {\mathcal Y}_{\mathcal V}$.
\end{enumerate}
Then there is:
\begin{enumerate}
\item \label{cond - ideal.rigid} a minimal rigid $F$-action on a real pretree~$T^*$ with trivial arc stabilizers, where
\item \label{cond - ideal.charrealpretree} the characteristic convex subset of~$T^*$ for~$\mathcal G$ is (an equivariant copy of)~${\mathcal Y}_{\mathcal V}$;
\item \label{cond - ideal.loxodromics} an element in~$F$ is $T^*$-loxodromic if and only if it is~$T$- or~${\mathcal Y}_{\mathcal V}$-loxodromic; and
\item \label{cond - ideal.automorphism}  $f^*\colon T^* \to T^*$ is the $\phi$-equivariant pretree-automorphism induced by~$f$ and~$h_{\mathcal V}$ --- the restriction of~$f^*$ to~$\mathcal Y_{\mathcal V}$ is~$h_{\mathcal V}$ and equivariantly collapsing~$\mathcal Y_{\mathcal V}$ recovers~$f$.
\end{enumerate}
Any pair $(T', f')$ satisfying the same conclusion admits an equivariant pretree-isomorphism $T' \to T^*$ that conjugates~$f'$ to~$f^*$;
moreover, if~$f$ is $F$-expanding, then so is~$f^*$.
\end{thm}
\begin{proof}
For each $v\in \mathcal V$, fix a set of $F$-orbit representatives~$\mathcal D_v$ for ${T}$-directions at~$v$.
The set~$\mathcal D_v$ is finite as well, again by the index inequality.
Since the $\phi$-equivariant pretree-automorphism~$f$ permutes the $F$-orbits of ${T}$-branch points and ${T}$-directions, it induces permutations $\beta \in \mathrm{Sym}(\mathcal V)$ and $\partial \in \mathrm{Sym}(\bigcup_{v \in \mathcal V} \mathcal D_v)$ with $\partial(d) \in \mathcal D_{\beta \cdot v}$ if~$d \in \mathcal D_v$.

By viewing~$\mathcal Y_{\mathcal V}$ as real pretrees, the minimal isometric $\mathcal G$-action is also a minimal rigid $\mathcal G$-action.
For any choice $\vec p = ( p_d \in \widehat Y_v: v \in \mathcal V, d \in \mathcal D_v)$ of points not fixed by a ${\mathcal Y}_{\mathcal V}$-loxodromic element in~$\mathcal G$, we can apply Proposition~\ref{prop - naive stitching} to construct a real pretree~$T^* =  T^*(\vec p \,)$ that satisfies Conditions~\ref{cond - ideal.rigid}-\ref{cond - ideal.loxodromics}.
To find the choice $\vec p$ that implies the remaining Condition~\ref{cond - ideal.automorphism},
we must use the pretree-automorphism $f\colon T \to T$ and expanding homothety $ h_{\mathcal V}\colon {\mathcal Y}_{\mathcal V} \to {\mathcal Y}_{\mathcal V}$.

\medskip
The maps $f$ and $h_{\mathcal V}$ induce a $\phi$-equivariance set-bijection $f^*\colon T^* \to T^*$:
\begin{itemize}
\item for $p \in {T} \setminus (F \cdot \mathcal V)$, define $f^*(\iota(p)) \defeq \iota(f(p))$; and
\item for $(x,p) \in F \times {\mathcal Y}_{\mathcal V}[\, \vec p\,]$, define $f^*(\iota(x,p)) \defeq \iota(\psi(x)x_v, \hat h_v(p))$, where the implicit choice $( x_v \in F: v\in \mathcal V)$ was used to define the restriction $\left.\phi\right|_{\mathcal G}$, and hence $h_{\mathcal V}$ and its extension~$\hat h_{\mathcal V}\colon \mathcal Y_{\mathcal V}[\,\vec p\,] \to \mathcal Y_{\mathcal V}[\,\vec p\,]$. 
\end{itemize}

For each representative $v \in \mathcal V$ and direction $d \in \mathcal D_v$, let $s_{\partial \cdot d} \in G_{\beta \cdot v}$ be an element such that $f(d) = x_v s_{\partial \cdot d} \cdot \partial(d)$;
the element~$s_{\partial \cdot d}$ is unique (trivial arc stabilizers).
The set-bijection~$f^*$ is a pretree-automorphism if and only if $\hat h_v(p_d) = s_{\partial \cdot d} \cdot p_{\partial \cdot d}$ for all $v \in \mathcal V$ and $d \in \mathcal D_v$ (see Appendix~\ref{SecSolns} for \hyperlink{thm - ideal stitch}{details}).

As $\bigcup_{v\in \mathcal V} \mathcal D_v$ is finite and $h_{\mathcal V}$ is an expanding homothety, the system of equations has a unique solution $ \vec p = (p_d \in \overline Y_v : v \in \mathcal V, d \in \mathcal D_v)$ in the metric completion of $\mathcal Y_{\mathcal V}$; these points are not fixed by any ${\mathcal Y}_{\mathcal V}$-loxodromic element in~$\mathcal G$.
Let $T^* = T^*(\vec p\,)$ be the real pretree given by this solution.
Then the induced function $f^*$ is a $\psi$-equivariant pretree-automorphism, as required for Condition~\ref{cond - ideal.automorphism}.
Any pair $(T', f')$ satisfying the same conclusion arises from the same construction and uniqueness follows from uniqueness of~$\vec p$.

\medskip
Finally, assume~$f$ is $F$-expanding. 
For all $x \in F$ and $n \ge 1$, we need to show the composition $\gamma^* \defeq x \circ (f^*)^n$ is either: 1) elliptic with a unique fixed point and it expands at this fixed point; or 2) loxodromic with an axis that is not shared with any $T^*$-loxodromic element in~$F$.

Let $\gamma \defeq  x \circ f^n \colon T \to T$, then 
we can use the $F$-expanding assumption on~$f$.

For the first case, assume~$\gamma$ is loxodromic with an axis that is not shared with any $T$-loxodromic element in~$F$.
Then by construction of~$T^*$ and Condition~\ref{cond - ideal.loxodromics},~$\gamma^*$ is loxodromic with an axis that is not shared with any~$T^*$-loxodromic element in~$F$.
For the rest of the proof, we may assume~$\gamma$ is elliptic with a unique fixed point~$p \in T$ and it expands at~$p$.

For the second case, suppose~$p$ has a trivial stabilizer.
Then $\gamma^*$ is elliptic with a unique fixed point~$\iota(p) \in T^*$ and it expands at~$\iota(p)$.

For the final case, suppose $p$ has a nontrivial stabilizer.
Then $p = y \cdot v$ for some $y \in F$ and $v \in \mathcal V$. 
In particular,~$v$ is fixed by $\beta^n \in \mathrm{Sym}(\mathcal V)$ since~$p = y \cdot v$ is fixed by $\gamma = x \circ f^n$.
In fact, $y \cdot v = x \circ f^n(y \cdot v) = x \psi^n(y)\psi^{n-1}(x_v)\cdots x_{\beta^{n-1} \cdot v} \cdot v$  implies \[s_v \defeq y^{-1}x\psi^n(y)\psi^{n-1}(x_v)\cdots x_{\beta^{n-1} \cdot v} \text{ is in } G_v,\]
and $\gamma^*(\iota(y, q)) = (y, s_v \cdot \hat h_{\beta^{n-1} \cdot v} \circ \cdots \circ \hat h_v(q))$ for all $q \in Y_v[\,\vec p\,]$.

Let $\gamma_v \defeq  s_v \circ h_{\beta^{n-1} \cdot v} \circ \cdots \circ h_v$. 
As $h_{\mathcal V}$ is an expanding homothety, the extension of the composition $\gamma_v$ to the metric completion has a unique fixed point $q_v \in \overline Y_v$ that is repelling.
If~$q_v \in Y_v[\,\vec p\,]$, then $\gamma^*$ is elliptic with a unique fixed point~$\iota(y,q_v)$ and it expands at~$\iota(y,q_v)$. 
Otherwise,~$q_v \notin Y_v[\,\vec p\,]$ and~$\gamma^*$ is loxodromic.
As~$\gamma$ is elliptic,~$\gamma^*$ does not share its axis (in~$T^*$) with any $T$-loxodromic element in~$F$.
As~$\iota(y, q_v)$ is an end of the axis for~$\gamma^*$ and~$q_v$ is in the metric completion~$\overline Y_v$,~$\gamma^*$ cannot share its axis with a conjugate (in~$F$) of a $Y_{v}$-loxodromic element of~$G_v$.
By Condition~\ref{cond - ideal.loxodromics},~$\gamma^*$ does not share an axis with any $T^*$-loxodromic element in~$F$ and we are done.
\end{proof}

\nonumsec{Epilogue --- are all limits the same?}\label{SecCanon}

As a preview for the sequel~\cite{Mut22}, we end the paper discussing whether the real pretree of Theorem~\ref{thm - limpretrees} is {\it canonical}.
There are two ways ``canonical'' can be interpreted: 
\begin{enumerate}
\item the real pretree produced by the proof of the theorem does not depend on any choices made in the proof --- in a sense, the real pretree was {\it predetermined};
\item limit pretrees (i.e.~pretrees satisfying the conclusion of the theorem) are equivariantly pretree-isomorphic.
\end{enumerate}
A priori, the second requirement seems stronger as, presumably, there might be limit pretrees that do not arise from the blow-up construction of the theorem.
While the jury is still out on the second interpretation, we prove the first one to be true in the sequel!

\medskip
A careful rereading of Theorem~\ref{thm - limpretrees}'s proof reveals that Theorem~\ref{thm - findtt} is the only source of indeterminacy: both blow-up constructions in Theorems~\ref{thm - simple patch} and~\ref{thm - ideal stitch} use uniquely determined attaching points!
In other words, once the hierarchy of expanding forests (produced by Proposition~\ref{prop - topforest}) is fixed, the real pretree given at the end of the proof is already determined.
Unfortunately, the irreducible train tracks (produced by Theorem~\ref{thm - findtt}) used in Proposition~\ref{prop - topforest} do depend on choices;
in general, there are no ``canonical'' train tracks and our main motivation was to address this issue.

To this end, we also suspect that Theorem~\ref{thm - limpretrees} can be strengthened by requiring the real pretree's completion to admit an equivariant surjection from the Cantor set boundary of the free group~$\partial F$ that is continuous with respect to the observers' topology.
First and foremost, this would immediately allow us to upgrade the rigid action on the real pretree to a convergence action on the completion.
Secondly, this would imply the completion of the limit pretree is dual to certain (relative) laminations in the free group (rel. the maximal elliptic subgroup system).
Using the dynamics of the automorphism acting on~$\partial F$, we might be able to show that these laminations, and hence their dual pretrees, are unique.
Alternatively, Bestvina--Feighn--Handel defined {\it topmost attracting laminations}~\cite[Section~6]{BFH00} that can be used to prove a canonical version of Proposition~\ref{prop - topforest}, thereby making the limit pretrees of Theorem~\ref{thm - limpretrees} canonical --- this is the approach taken in the sequel.

\medskip
Fortunately, there is already precedent for these approaches.
When the free group outer automorphism is {\it irreducible} with infinite order, Bestvina--Feighn--Handel proved that a limit tree produced by the proof of Proposition~\ref{prop - topforest} is unique up to equivariant homothety~\cite[Lemma~3.4]{BFH97}.
In particular, {\it the} limit tree is independent of the chosen train track used to define it.
In fact, the limit trees are canonical in a stronger sense: all trees satisfying the conclusion of Proposition~\ref{prop - topforest} are equivariantly homothetic.
Importantly, they do not assume beforehand that these trees are projective limits of iterating a topological representative.
Later, Levitt--Lustig proved the corresponding observers' compactifications of these trees admit canonical equivariant quotient maps from the boundary~$\partial F$~\cite{LL03}.

For an irreducible atoroidal automorphism, Kapovich--Lustig~\cite[Proposition~4.5]{KL15} recently showed that the {\it dynamical} lamination dual to the limit tree is exactly the {\it geometric} lamination defined by Mahan Mj (formerly Mitra)~\cite[p.~388]{Mit97} using the hyperbolic geometry of the corresponding mapping torus~$F \rtimes \mathbb Z$.

\medskip
When we drop irreducibility but keep the atoroidal assumption, there still are promising results. 
Following works of Mj and Bowditch, Elizabeth Field has recently used the hyperbolicity of the mapping torus to construct a canonical dendrite that is dual to Mj's geometric lamination~\cite{Fie20}. 
We note that the index of the dendrite's interior is an invariant of the atoroidal outer automorphism.
It might follow from the construction that Field's dendrite is a completion of a limit pretree.
On the other hand, Uyanik has proven that atoroidal automorphisms act with generalized north-south dynamics on the projective space of currents~\cite[Theorem~1.4]{Uya19}.
As an extension of Kapovich--Lustig's result, the geometric lamination should be the support of the repelling simplex of projective currents!

\medskip
Finally, when we drop the atoroidal assumption, it seems reasonable to conjecture that there are generalizations of Field's and Uyanik's results that use relative hyperbolicity and relative currents respectively, and these can be unified.



\appendix
\section{Index theory for small rigid actions}\label{SecIndexing}
In this appendix, we sketch Gaboriau--Levitt's proof of the index inequality~\cite{GL95}.
Although the inequality was originally stated for isometric actions, their proof only used the rigidity of actions ---  the metric plays no other important role.
Thus, we sketch this general version of the index theory.

\medskip
A rigid $F$-action on a real pretree~$T$ is {\it small} if arc stabilizers are cyclic.
Fix a minimal small rigid $F$-action on a real pretree~$T$.
For each orbit of points $[p] \in F\backslash T$, let $G_p \le F$ be the stabilizer for $p \in T$ and $\#\mathrm{dir}_1[p]$ denote the number of $G_p$-orbits of directions at~$p$ with trivial stabilizers.
Recall that the complexity of~$G_p$ is $c(G_p) = 2 \cdot \mathrm{rank}(G_p) - 1$.
The index at~$[p]$ is:
\[ i[p] \defeq c(G_p) - 1 + \#\mathrm{dir}_1[p] \quad \in ~ \mathbb Z_{\ge 0} \cup \{ \infty \}. \]
Nonnegativity follows from minimality of the action ($F$ is not trivial). 
The \underline{index} is:
\[ i(F \backslash T) \defeq \sum_{[p]\in F\backslash T} i[p] \quad \in ~ \mathbb Z_{\ge 0} \cup \{ \infty \}. \]

Gaboriau--Levitt proved the following:

\begin{restate}{Theorem}{athm:index-inequality}[cf.~{\cite[Theorem~III.2]{GL95}}]
\( i(F \backslash T) < c(F).\)
\end{restate}

\subsection*{\textbullet ~The real pretree associated to a rigid system}

A real pretree is \underline{finite} if it is the {\it convex hull} of a finite subset.
A function $j\colon T \to T'$ of real pretrees is a \underline{morphism} if every closed $T$-interval is a finite union of closed subintervals that are pretree-embedded into~$T'$ by~$j$.

Let $K$ be a finite real pretree and $x_i\colon A_i \to B_i~(1 \le i \le n)$ be pretree-isomorphisms of finite real pretrees in~$K$.
Alternatively, $x_i \colon K \to K$ are {\it partial} pretree-automorphisms whose domains are convex and complete.
Denote this system by~$\mathcal K$.
A system $\mathcal K = (K, \{ x_i \})$ is \underline{nontrivial} if no domain $A_i \defeq \mathrm{dom}(x_i)$ is empty.

Let~$F_{\mathcal K}$ be free group generated by the set $\{ x_i \}_{i=1}^n$.
Note that any element~$x \in F_{\mathcal K}$ (i.e. reduced word in~$x_i^{\pm 1}$) determines a partial pretree-automorphism $x\colon K \to K$ with a possibly empty domain.
The length of $x \in F_{\mathcal K}$ as a reduced word in~$x_i^{\pm 1}$ is denoted~$|x|_{\mathcal K}$.

\begin{athm}[cf.~{\cite[Theorem~I.1]{GL95}}]\label{athm:assoc-pretree} For a nontrivial system~$\mathcal K = (K, \{x_i\})$, there is:
\begin{enumerate} 
\item\label{athm:assoc-pretree:exists} an $F_{\mathcal K}$-action on a real pretree~$T_{\mathcal K}$ by pretree-automorphisms;
\item\label{athm:assoc-pretree:embeds} a pretree-embedding $\iota\colon K \to T_{\mathcal K}$ with $x_i \cdot \iota(a) = \iota(x_i(a))$ for $a \in A_i~(1 \le i \le n)$; and
\item\label{athm:assoc-pretree:covers} the $F_{\mathcal K}$-orbit of $\iota(K)$ covers~$T_{\mathcal K}$, i.e.~$T_{\mathcal K} = F_{\mathcal K} \cdot \iota(K)$.
\end{enumerate}
If another pretree $F_{\mathcal K}$-action on a real pretree~$T'$ satisfies Condition~\ref{athm:assoc-pretree:embeds} with $\iota'\colon K \to T'$, then there is a unique equivariant morphism $j\colon T_{\mathcal K} \to T'$ with $j(\iota(a)) = \iota'(a)$ for $a \in K$.
\end{athm}
\noindent The last part states that the pair of the pretree $T_{\mathcal K}$ and $F_{\mathcal K}$-action is universal with respect to Condition~\ref{athm:assoc-pretree:embeds}.
The real pretree~$T_{\mathcal K}$ (along with its $F_{\mathcal K}$-action and pretree-embedding $\iota\colon K \to T_{\mathcal K}$) will be referred to as the system's \underline{associated real pretree}.
\begin{proof}[Sketch of proof]
The set~$T_{\mathcal K}$ is defined through the quotient
\[ \pi\colon F_{\mathcal K} \times K \to T_{\mathcal K},\]
given by the identifications:
\begin{itemize}
\item for each $a \in A_i$ and $x \in F_{\mathcal K}$, $(x x_i, a)$ is identified with $(x, x_i(a))$.
\end{itemize}
$T_{\mathcal K}$ inherits a set $F_{\mathcal K}$-action from the left-multiplication on the first factor of $F_{\mathcal K} \times K$. 
We need to define a pretree structure on~$T_{\mathcal K}$.

Let the length of $(x, a) \in F_{\mathcal K} \times K$ be~$|x|_{\mathcal K}$.
Every point $p \in T_{\mathcal K}$ has a unique shortest element $\delta(p) \defeq (x, a) \in \pi^{-1}(p)$.
Suppose $p, q \in T_{\mathcal K}$ have representatives $\delta(p) = (x, a)$ and $\delta(q) = (y, b)$.

\medskip
\noindent \underline{Case 1: $x = y$}. 
Define $[p,q] \defeq \pi(x, [a,b]_K)$.

\medskip
\noindent \underline{Case 2: $|y^{-1}x|_{\mathcal K} = m \ge 1$}. 
Let $y^{-1} x \in F_{\mathcal K}$ be the reduced word $y_m \cdots y_1$ in $x_i^{\pm 1}$.
Set~$c_1$ to be the {\it projection} of~$a$ to the domain $\mathrm{dom}(y_1) \neq \emptyset$.
Assume~$c_j \in \mathrm{dom}(y_j)$ for $1 \le j < m$, and set~$c_{j+1}$ to be the projection of~$y_j(c_j)$ to $\mathrm{dom}(y_{j+1}) \neq \emptyset$.
Define
\[ [p,q] \defeq \pi(x, [a,c_1]_K) \cup \left( \bigcup_{j=1}^{m-1} \pi(x y_1^{-1} \cdots y_j^{-1}, [y_j(c_j), c_{j+1}]_K) \right) \cup \pi(y, [y_m(c_m), b]_K). \]
We leave it as an exercise to check that $(T_{\mathcal K}, [\cdot, \cdot])$ is a real pretree and $F_{\mathcal K}$ acts on~$T_{\mathcal K}$ by pretree-automorphisms.
This concludes Theorem~\ref{athm:assoc-pretree}(\ref{athm:assoc-pretree:exists}).

Let $1 \in F_{\mathcal K}$ be the identity element.
The map $\iota \colon K \to T_{\mathcal K}$ given by $a \mapsto \pi(1, a)$ is a pretree-embedding (see Lemma~\ref{alem:action-condition} below).
If $a \in A_i$, then by construction
\[ x_i \cdot \iota(a) = \pi(x_i, a) = \pi(1, x_i(a)) = \iota(x_i(a)).\]
Also by construction, $F_{\mathcal K} \cdot \iota(K) = \pi(F_{\mathcal K} \times K) = T_{\mathcal K}$.
This gives us Theorem~\ref{athm:assoc-pretree}(\ref{athm:assoc-pretree:embeds}-\ref{athm:assoc-pretree:covers}).

Finally, we also leave the proof of the universal property as an exercise:
use the fact that any closed interval of~$T_{\mathcal K}$ decomposes (by construction) into finitely many closed subintervals, each contained in a translate of~$\iota(K)$.
\end{proof}

\begin{alem}[cf.~{\cite[Proposition~I.4]{GL95}}]\label{alem:action-condition} 
Let $\mathcal K$ be a nontrivial system and $T_{\mathcal K}$ be the associated real pretree. For $y \in F_{\mathcal K}$ and $a,b \in K$, 
\[ y \cdot \iota(a) = \iota(b) \text{ if and only if } y(a) = b. \]
\end{alem}
\begin{proof}[Sketch of proof]
The identifications on $F_{\mathcal K} \times K$ used to define $T_{\mathcal K}$ are symmetric but not transitive.
Since the reduced word in~$x_i^{\pm 1}$ representing an element of~$F_{\mathcal K}$ is unique,
the identifications generate the following equivalence relation: for $a, b \in K$ and $x, y \in F_{\mathcal K}$, $(xy, a) \sim (x, b)$ if and only if $y(a) = b$.
\end{proof}

\begin{acor}[cf.~{\cite[Proposition~I.5]{GL95}}]\label{acor:fixed-injective}
Let $\mathcal K$ be a nontrivial system and $T_{\mathcal K}$ be the associated real pretree.
For any nontrivial $y \in F_{\mathcal K}$ and any morphism $j\colon T_{\mathcal K} \to T'$ as in the last part of Theorem~\ref{athm:assoc-pretree}, the restriction of~$j$ to the fixed-point set~$\mathrm{Fix}_{T_{\mathcal K}}(y)$ is injective.
\end{acor}
\begin{proof}[Sketch of proof]
As~$j$ is equivariant, $\iota'$ is injective, and $j(\iota(a)) = \iota'(a)$ for $a \in K$, it is enough to show that $\mathrm{Fix}_{T_{\mathcal K}}(y)$ is contained in some translate of $\iota(K)$.
Without loss of generality, we assume $y \in F_{\mathcal K}$ is a cyclically reduced nonempty word in~$x_i^{\pm 1}$ and show $\mathrm{Fix}_{T_{\mathcal K}}(y)$ is contained in $\iota(K)$ --- or rather, we prove the contrapositive.

Let $p \in T_{\mathcal K}$ be fixed by a nontrivial~$y \in F_{\mathcal K}$ and assume $p \notin \iota(K)$, i.e.~$\delta(p) = (x,a)$ has positive length.
Since $(yx, a)$ and $(x,a)$ represent~$p$, we get $x^{-1}yx(a) = a$ by Lemma~\ref{alem:action-condition}.
Let $z \defeq x^{-1}yx \in F_{\mathcal K}$ be the reduced nonempty word $z_l \cdots z_1$ in~$x_i^{\pm 1}$.
Then $(xz_1^{-1}, z_1(a))$ and $(xz_l, z_l^{-1}(a))$ also represent~$p$.
By minimality of $\delta(p) = (x,a)$, the element $x \in F_{\mathcal K}$ as a reduced word in~$x_i^{\pm 1}$ cannot end with~$z_1$ nor~$z_l^{-1}$.
So $y = x z x^{-1}$ as a reduced word in~$x_i^{\pm 1}$ is not cyclically reduced.
\end{proof}

A partial pretree-automorphism is \underline{rigid} if either its fixed-point set is empty or it fixes no direction at its fixed-point set.
A nontrivial system~$\mathcal K$ is \underline{rigid} if every element of~$F_{\mathcal K}$ determines a rigid partial pretree-automorphism of~$K$.

\begin{alem}\label{alem:action-rigidity}
If~$\mathcal K$ is a rigid system, then the pretree $F_{\mathcal K}$-action on the associated real pretree~$T_{\mathcal K}$ is rigid.
\end{alem}
\begin{proof}
Suppose $x \in F_{\mathcal K}$ has fixed points in $T_{\mathcal K}$ and let~$d$ be an arbitrary $T_{\mathcal K}$-direction at $\mathrm{Fix}_{T_{\mathcal K}}(x)$.
Let $p_d \in \mathrm{Fix}_{T_{\mathcal K}}(x)$ be an attaching point for~$d$.
By construction, some nondegenerate interval $[p_d,q] \subset \{ p_d \} \cup d$ is contained in a translate $y \cdot \iota(K)$ for some $y \in F_{\mathcal K}$.
Let $p_d = y \cdot \iota(a_d)$ and $q = y \cdot \iota(b)$ for some $a_d,b \in K$.

For any $p' = y \cdot \iota(a') \in [p_d,q]$, Lemma~\ref{alem:action-condition} implies $x \cdot p' = p'$ if and only if $y^{-1}xy(a') = a'$.
In particular,~$a_d \in \mathrm{Fix}_K(y^{-1}xy)$.
Since no point in~$d$, and hence $(p_d,q] = y \cdot \iota((a_d,b]_K)$, is fixed by~$x$, no point in $(a_d,b]_K$ is fixed by the partial pretree-automorphism $y^{-1}xy$.
So $(a_d,b]_K$ determines a $K$-direction~$d_K$ at~$\mathrm{Fix}_K(y^{-1}xy)$.

By rigidity of the system, $y^{-1}xy$ (partially applied to $K$) does not fix $d_K$;
therefore, the element~$x\in F_{\mathcal K}$ (acting on $T_{\mathcal K}$) does not fix~$d$ (Lemma~\ref{alem:action-condition}) and we are done.
\end{proof}

Assume the system~$\mathcal K$ is rigid and~$T_{\mathcal K}$ is the associated real pretree.
Let $S \subset K$ be the finite subset of vertices (i.e.~branch points or {\it endpoints}) in $K$, $A_i$, and $B_i~(1 \le i \le n)$.

\begin{alem}[cf.~{\cite[Proposition~I.8]{GL95}}]\label{alem:branches-stabs} If $p \in T_{\mathcal K}$ is a branch point, then the $F_{\mathcal K}$-orbit of~$p$ contains a point of $\iota(S)$ and there are finitely many $G_p$-orbits of directions at~$p$.
\end{alem}
\begin{proof}[Sketch of proof]
Suppose $[p,q] \cap \iota(K) = \{ p \}$ and let $p = \iota(a)$ with $a \in K$.
Let $y \in F_{\mathcal K}$ be the shortest reduced word $y_m \cdots y_1$ in $x_i^{\pm 1}$ whose translate $y \cdot \iota(K)$ covers a nondegenerate subinterval $[p, p']$.
So $[p, p'] = y \cdot \iota([b, b']_K)$ for some $b, b' \in K$ and $y(b) = a$ (Lemma~\ref{alem:action-condition}).
Minimality of $y$ and $[p,p'] \cap \iota(K) = \{ p \}$  imply $[b, b']_K \cap \mathrm{dom}(y_1) = \{ b \}$.
Then $b \in S$.

Deduce that the $F_{\mathcal K}$-orbit of any branch point in~$T_{\mathcal K}$ intersects~$\iota(S)$. Similarly, the $F_{\mathcal K}$-orbit of any direction at a branch point contains an ``end'' of $\iota(K \setminus S)$; therefore, the number of $G_p$-orbits of directions at~$p$ is at most the number of ends of $\iota(K) \setminus \left(\iota(S) \cap (F_{\mathcal K} \cdot p)\right)$.
\end{proof}

\subsection*{\textbullet ~Geometric actions}

Fix a basis $\{ x_1, \ldots, x_n \}$ for the free group $F$. 
Let~$T$ be a real pretree with a rigid $F$-action.
For any finite real pretree $K \subset T$, we get a system $\mathcal K = (K, \{ x_i \})$ by letting~$x_i \in F$ partially act on~$K$ as the partial pretree-automorphism given by restricting the action of~$x_i$ on~$T$ to the domain $\left( x_i^{-1}\cdot K \right) \cap K$.
The free groups $F_{\mathcal K}$ and $F$ will be identified as they have the ``same'' basis~$\{ x_i \}$.

Choose a finite real pretree $K \subset T$ so that the system $\mathcal K$ is nontrivial, in which case it will be a rigid system.
By Theorem~\ref{athm:assoc-pretree} and Lemma~\ref{alem:action-rigidity}, the rigid system~$\mathcal K$ has an associated real pretree, denoted~$T_{K}$, with a rigid $F$-action.
By Lemma~\ref{alem:action-condition}, the rigid $F$-action on~$T_K$ is small if the rigid $F$-action on~$T$ was small.

\medskip
Now suppose the $F$-action on (a nondegenerate)~$T$ is minimal and pick a basepoint $p_0 \in T$. By minimality of the rigid $F$-action on~$T$, the point~$p_0$ is in the $T$-axis for some loxodromic element $x \in F$.
For any integer $m$, define $K_m$ to be the convex hull of the partial orbit $\{ y \cdot p_0 : |y| \le m \}$ and let $T_m$ be the real pretree associated to $(K_m, \{ x_i \})$.
Choose any $m \ge |x|$, then the interval $[x^{-1} \cdot p_0, p_0] \subset K_m$ is contained in the domain of $x \colon K_m \to K_m$ (as the composition of partial pretree-automorphisms $x_i \colon K_m \to K_m$).
Since $x([x^{-1} \cdot p_0, p_0]) = [p_0, x \cdot p_0]$,
we deduce~$x$ is $T_m$-loxodromic and its axis in~$T_m$ contains~$\iota_m(p_0)$.
Consequently, the rigid $F$-action on~$T_m$ is minimal.

Using the universal property in Theorem~\ref{athm:assoc-pretree}, the chain $K_1 \subset K_2 \subset \cdots \subset T$ gives us a direct system of equivariant morphisms $T_1 \to T_2 \to \cdots \to T$.
The real pretree~$T$ is a \underline{strong limit} of this  direct system:
each closed interval $I \subset T_1$ has an image under the morphism $T_1 \to T_{m}$~(for large enough $m$) that is embedded into~$T$ by the morphism $T_{m} \to T$.
For the proof, pick a closed interval $I \subset T_1$;
its image under the morphism $j_1\colon T_1 \to T$ is a finite real pretree, which is contained in~$K_m$ for large~$m$;
therefore, $\iota_m(j_1(I)) \subset \iota_m(K_m)$, the image of~$I$ under $T_1 \to T_m$, is embedded into~$T$ (with image~$j_1(I)$) by $T_m \to T$ as the latter restricts to a pretree-embedding of~$\iota_m(K_m)$.

\begin{alem}[cf.~{\cite[Proposition~II.1]{GL95}}]\label{alem:geometric-action} Let~$T$ be a real pretree with a minimal rigid $F$-action. The following conditions are equivalent:
\begin{enumerate}
\item\label{alem:geometric-action:geometric} $T$ is equivariantly pretree-isomorphic to~$T_{\mathcal K}$ for some rigid system $\mathcal K = (K, \{ x_i \})$; and
\item\label{alem:geometric-action:trivial} $T$ can only be a strong limit in a trivial way: if $T$ is a strong limit of a direct system of equivariant morphisms $T_1 \to T_2 \to \cdots \to T$ with minimal rigid $F$-actions on~$T_j$, then $T_{m} \to T$ is an equivariant pretree-isomorphism for some $m \ge 1$.
\end{enumerate}
\end{alem}

\noindent A real pretree~$T$ with a minimal rigid $F$-action is \underline{geometric} if these conditions hold.
By the second characterization, ``geometricity'' will be independent of the choice of basis for~$F$.

\begin{proof}[Sketch of proof] {~}

$(\ref{alem:geometric-action:trivial} \implies \ref{alem:geometric-action:geometric})$: 
By the preceding discussion, $T$ is a strong limit of a direct system of equivariant morphisms $T_1 \to T_2 \to \cdots \to T$ and the rigid $F$-action on $T_j$ is minimal for large enough~$j$.
Since we are assuming strong limits are trivial, we have $T_m \to T$ is an equivariant pretree-isomorphism for some large~$m$.
Recall that $T_m$ is associated to a rigid system given by a finite real pretree $K_m \subset T$ and we are done.

$(\ref{alem:geometric-action:geometric} \implies \ref{alem:geometric-action:trivial})$:
Let $\rho\colon T_{\mathcal K} \to T$ be an equivariant pretree-isomorphism for some rigid system $\mathcal K = (K, \{ x_i \})$.
In particular, $T$ satisfies the universal property in Theorem~\ref{athm:assoc-pretree} with respect to the pretree-embedding $\rho \circ \iota \colon K \to T$.
We need to show that an arbitary strong limit $T_1 \to T_2 \to \cdots \to T$ of minimal rigid actions is trivial.
As a strong limit of minimal actions, we can lift the union of $\rho(\iota(K)) \subset T$ and its translates $x_i \cdot \rho(\iota(K)) \subset T ~(1 \le i \le n)$ to a pretree-isomorphic copy in $T_m$ for some large enough $m$.
By the universal property, there is a unique equivariant morphism $T \to T_m$.
So the equivariant morphism $T_m \to T$ (from the strong limit) must be a pretree-isomorphism and hence the strong limit is trivial.
\end{proof}

\subsection*{\textbullet ~Counting branch points}

We are now ready to prove the index inequality.







\begin{athm}[cf.~{\cite[Theorem~III.2]{GL95}}]\label{athm:index-inequality} For a minimal small rigid $F$-action on a real pretree~$T$,

\begin{enumerate}
\item\label{athm:index-inequality:geometric} \( i(F \backslash T) = c(F) - 1\) if~$T$ is geometric, and
\item\label{athm:index-inequality:otherwise} \( i(F \backslash T) < c(F) - 1\) otherwise.
\end{enumerate}
\end{athm}

\begin{proof}[Proof of Part~\ref{athm:index-inequality:geometric}]
Suppose $\mathcal K = (K, \{ x_i\colon A_i \to B_i \})$ is a rigid system and the rigid $F$-action on the associated real pretree $T_{\mathcal K}$ is small and minimal.
For a finite real pretree~$H$, the valence of $h \in H$ --- denoted $\nu_H(h)$ --- is the number of directions at $h$.
The valences satisfy the identity:
\(\sum_{h \in H}\left( \nu_H(h) - 2 \right) = -2.\)

Fix a point $p \in T_{\mathcal K}$.
Then $p = x \cdot \iota(a)$ for some $x \in F$ and $a \in K$ by Theorem~\ref{athm:assoc-pretree}(\ref{athm:assoc-pretree:covers}).
Define $V_p \defeq \{\, b \in K : y \cdot \iota(a) = \iota(b) \text{ for some } y \in F \,\}$.
The ``Cayley graph''~$\mathcal O_p$ is the oriented labelled graph with vertex set~$V_p$ and an oriented labelled edge $b \overset{x_i}\to c$ if $x_i(b) = c$.
The Cayley graph is connected (Lemma~\ref{alem:action-condition}).

Each oriented labelled edge has a weight $w{\left(b \overset{x_i}\to c\right)} \defeq \nu_{A_i}(b)$.
The ``blow-up''~$\mathcal O_p'$ is constructed by replacing: 1) each vertex $b \in V_p$ with $\nu_K(b)$-many vertices corresponding to the $K$-directions at $b$; and 2) each edge $b \overset{x_i}\to c$ with $w{\left(b \overset{x_i}\to c\right)}$-many oriented labelled edges in the ``obvious way.''
Let $\pi\colon \mathcal O_p' \to \mathcal O_p$ be the natural finite-to-one projection.

\begin{alem}[cf.~{\cite[Lemma~III.5]{GL95}}]\label{alem:index-inequality} Let~$G_p \le F$ be the stabilizer of $p \in T_{\mathcal K}$.
\begin{enumerate}
\item\label{alem:index-inequality:vertexgroupiso} $\pi_1(\mathcal O_p) \cong G_p$ (natural group isomorphism).
\item\label{alem:index-inequality:dirs} $\pi_0(\mathcal O_p') \cong G_p\backslash \pi_0(T_{\mathcal K} \setminus \{ p \})$ (natural set-bijection).
\item\label{alem:index-inequality:dirgroupiso} for $\mathcal O_d \in \pi_0(\mathcal O_p')$ corresponding to a $T_{\mathcal K}$-direction $d$ at $p$, $\pi_1(\mathcal O_d) \cong \mathrm{Stab}_{G_p}(d)$.
\end{enumerate}
\end{alem}
\begin{proof}[Proof sketch for lemma]
There is a natural injective homomorphism $\pi_1(\mathcal O_p, a) \to G_{\iota(a)}$ given by ``reading the oriented labels.''
This homomorphism is surjective (Lemma~\ref{alem:action-condition}) --- as $G_{\iota(a)} = x^{-1} G_p x$, this concludes part~\ref{alem:index-inequality:vertexgroupiso}.

Distinct components of~$\mathcal O_p'$ correspond to distinct $F$-orbits of $T_{\mathcal K}$-directions at translates of $p$ by Lemma~\ref{alem:action-condition}.
So there is a natural injective set-function \(\pi_0(\mathcal O_p') \to G_p \backslash \pi_0(T_{\mathcal K} \setminus \{ p \}).\)
The proof of Lemma~\ref{alem:branches-stabs} shows that this function is surjective --- this concludes part~\ref{alem:index-inequality:dirs}.

As in the first part, we have an injective homomorphism $\pi_1(\mathcal O_d, d_b) \to \mathrm{Stab}_{G_{\iota(b)}}(d_b)$, where $b \in V_p$;
in particular, $G_{\iota(b)} = x^{-1} G_p x$ and $\mathrm{Stab}_{G_{\iota(b)}}(d_b)$ is conjugate in~$F$ to a stabilizer in~$G_p$ of a $T_{\mathcal K}$-direction at~$p$.
This natural homomorphism is surjective (Lemma~\ref{alem:action-condition}) and we are done.
\end{proof}

Recall that~$S \subset K$ consists of the vertices of $K$, $A_i$, and $B_i$;
it is a finite set.
Let $\mathcal G \subset \mathcal O_p$ be a finite connected subgraph containing all vertices in~$S$ and all edges of weight~$\neq 2$.
Set $\mathcal G' \defeq \pi^{-1}(\mathcal G) \subset \mathcal O_p'$;
this is a finite subgraph as well since~$\pi$ was finite-to-one.

By Lemma~\ref{alem:branches-stabs} and Lemma~\ref{alem:index-inequality}(\ref{alem:index-inequality:dirs}-\ref{alem:index-inequality:dirgroupiso}), $\mathcal O_p'$ has finitely many components, each with rank~0 or~1 (small rigid action).
Enlarge $\mathcal G$ while maintaining finiteness and assume $\mathcal G'$ supports the fundamental group of each component of $\mathcal O_p'$.

For any finite connected graph, we have the Euler characteristic identity:
\[ 1 - \mathrm{rank}(\pi_1(\,\cdot\,)) = \# V(\,\cdot\,) - \# E(\,\cdot\,),\]
where $V(\,\cdot\,)$ and $E(\,\cdot\,)$ denote the vertex and edge sets respectively. 
Sum over the finitely many components of $\mathcal G_j' \subset \mathcal G'$;
 we get
\[ \sum_j \left( 1 -\mathrm{rank}(\pi_1(\mathcal G_j')) \right) = \sum_{b \in V(\mathcal G)} \nu_K(b) - \sum_{e \in E(\mathcal G)} w(e).\]
Subtract the double of the Euler characteristic identity applied to~$\mathcal G$:
\begin{equation}\label{aeq:index-inequality:total}
\begin{aligned}
2\cdot \mathrm{rank}(\pi_1(\mathcal G)) - 2 +  \sum_j \left( 1 -\mathrm{rank}(\pi_1(\mathcal G_j')) \right) &= \sum_{b \in V(\mathcal G)}(\nu_K(b) - 2) - \sum_{e \in E(\mathcal G)}(w(e) - 2) \\
&= \sum_{b \in V_p}(\nu_K(b) - 2) - \sum_{e \in E(\mathcal O_p)}(w(e) - 2).
\end{aligned}
\end{equation}
The last equality follows from the fact~$\mathcal G$ was assumed to contain all vertices and edges with nonzero summand contributions to the right-hand side.
Note that the right-hand side is finite and independent of~$\mathcal G$.
The summands on the left-hand side are nonnegative;
therefore, $\pi_1(\mathcal G)$ has uniformly bounded rank and $\pi_1(\mathcal O_p)$ is finitely generated.

Assume~$\mathcal G$ supports the fundamental group of~$\mathcal O_p$.
Then each component of~$\mathcal O_p'$ has at most one component of~$\mathcal G'$.
By Lemma~\ref{alem:index-inequality}, Equation~(\ref{aeq:index-inequality:total}) can be rewritten as
\[i[p] = \sum_{b \in V_p} (\nu_K(b) - 2) - \sum_{i=1}^n \sum_{b \in V_p \cap A_i} (\nu_{A_i}(b) - 2).\]
Summing over all $F$-orbits $[p] \in F \backslash T_{\mathcal K}$ and applying the valence identity in the first paragraph of the proof, we get
\[\begin{aligned} i(F \backslash T_{\mathcal K}) &= \sum_{b \in K} (\nu_K(b) - 2) - \sum_{i=1}^n \sum_{b \in A_i} (\nu_{A_i}(b) - 2 ) \\
&= -2 + 2n \\
&= c(F) - 1.
\end{aligned}\]
This concludes the geometric case of the theorem.
\end{proof}
\begin{proof}[Proof of Part~\ref{athm:index-inequality:otherwise}]
Suppose~$T$ is a nontrivial strong limit of a direct system of equivariant morphisms $T_1 \to T_2 \to \cdots \to T$ where each~$T_i$ is geometric.

Fix a branch point $p \in T$, then choose $g_1, \ldots, g_r \in G_p$ that freely generate a free factor of $G_p$ and directions $d_1, \ldots, d_s$ at $p$ that have trivial stabilizers and are in distinct $G_p$-orbits.
As a strong limit, we can lift the point~$p$ to $p' \in T_m$ and the directions $d_j~(1 \le j \le s)$ to~$d_j'$ in~$T_m$ so that $p'$ is fixed by all $g_k~(1 \le k \le r)$ for large enough $m \ge 1$.
By the equivariance of the morphism $T_m \to T$, the $T_m$-directions $d_j'$ have trivial stabilizers and $G_{p'} \le G_p$ contains the free factor generated by $g_1, \ldots, g_r$.
So $2r - 2 + s \le i[p']$.
From Part~(\ref{athm:index-inequality:geometric}), $i[p']$ is uniformly bounded by $\iota(F \backslash T_m) = c(F)-1$.
Thus~$G_p$ is finitely generated, there are finitely many $G_p$-orbits of directions with trivial stabilizers at $p$, and $i[p] \le i[p']$.

Similarly, any finite set of branch points $p_1, \ldots \in T$ in distinct $F$-orbits can be lifted to branch points $p_1', \ldots \in T_m$ with $i[p_j] \le i[p_j']$ for large enough $m$.
By equivariance, the branch points $p_1', \ldots$ are in distinct $F$-orbits.
So we get $i(F \backslash T) \le c(F)-1$ from Part~(\ref{athm:index-inequality:geometric}).

\medskip
To show the inequality $i(F \backslash T) < c(F) - 1$, we assume the equality $i(F \backslash T) = c(F) - 1$ and deduce a contradiction.
Suppose each~$T_m$ is associated to a finite real pretree $K_m \subset T$ that is the convex hull of some partial $F$-orbit of a branch point $p_0 \in T$.

For large $m \ge 1$, $K_m$ intersects each $F$-orbit~$[p]$ of points in~$T$ with $i[p] > 0$ and $p' \in \iota_m(K_m \cap [p]) \subset T_m$ has index $i[p'] = i[p]$.
Conversely, as $i(F \backslash T_m) = i(F \backslash T)$, every $F$-orbit~$[p']$ of points in~$T_m$ with $i[p'] > 0$ intersects $\iota(K_m \cap [p])$ for some $F$-orbit of points in~$T$ with $i[p] = i[p']$.
Fix a large enough~$m \ge 1$.

We finally use the hypothesis that~$T$ is a nontrivial strong limit of $T_1 \to T_2 \to \cdots \to T$; i.e.~the equivariant morphism $j_m\colon T_m \to T$ is not injective.
Thus, two distinct directions $d_1, d_2$ at some point $q' \in T_m$ are mapped to the same direction at $q \defeq j_m(q') \in T$.

If $i[q'] > 0$, then $i[q] = i[q']$ implies $d_1, d_2$ have nontrivial stabilizers.
Their image in~$T$ also has a nontrivial cyclic stabilizer and so the union of their stabilizers generate a cyclic subgroup;
therefore, some nondegenerate arc in $d_1 \cup \{ q' \} \cup d_2$ intersecting both directions is fixed by a nontrivial element $y \in F$.
But Corollary~\ref{acor:fixed-injective} states that~$j_m$ is injective on the fixed-point set of a nontrivial element --- a contradiction.

If~$q'$ is a branch point with $i[q']=0$, then again $d_1, d_2$ have nontrivial stabilizers and the argument is the same.
If~$q'$ is not a branch point, then some nondegenerate arc in $d_1 \cup \{ q' \} \cup d_2$ intersecting both directions is contained some translate $y \cdot \iota_m(K_m)$ --- this is where we use the assumption $K_m$ was the convex hull of some branch points. 
We have another contradiction as~$j_m$ is injective on $\iota_m(K_m)$ and we are done.
\end{proof}

\section{Solutions to some exercises}\label{SecSolns}
\hypertarget{prop - growth}{}

\hypertarget{prop - naive stitching}{}
We now state elementary facts about pretrees whose proofs are left to the reader.
Let $p,q,r$ be arbitrary points in a pretree $(T, [\cdot, \cdot])$.

\medskip
\noindent \textbullet ~(subinterval) If $q \in [p,r]$, then $[p,q] \subset [p,r]$.


\noindent \textbullet ~(arc-overlaps) If $[p,q] \cap [p,r]$ is not a singleton, then it is an arc.
%

\noindent \textbullet ~(pseudocenters) If $m,n \in [p,q] \cap [q,r] \cap [p,r]$, then $m = n$.


\medskip
\begin{proof}[Checking pretree axioms are satisfied in Proposition~\ref{prop - naive stitching}]
Let $p^*,q^*,r^* \in T^*$.

\medskip
\noindent \textbullet ~Symmetric (axiom~\ref{axiom - symmetry}): 

Suppose $\pi(p^*) = \pi(q^*)$. If $\pi(p^*) \notin H \cdot \mathcal V$, then $p^* = q^*$ and the axiom holds trivially. 
If $\pi(p^*) \in F \cdot \mathcal V$, then $p^* = \iota(y,p)$ and $q^* = \iota(y,q)$ for some $y \in H$, $v \in \mathcal V$, and $p,q \in \widehat T_v$.
The axiom follows from symmetry of $[\cdot, \cdot]_v$.

We may now assume $\pi(p^*) \neq \pi(q^*)$. 
By definition, the $\pi$-preimages of the open interval $(\pi(p^*), \pi(q^*))$ in $[p^*, q^*]^*$ and $[q^*, p^*]^*$ are the same subset of $T^*$ as they have the same {\it unordered $\mathcal V$-turns} in $[\pi(p^*), \pi(q^*)]$: use symmetry of $[\cdot, \cdot]$ and $[\cdot, \cdot]_v~(v \in \mathcal V)$.
Finally, the $\pi$-preimages of $\pi(p^*)$ in $[p^*, q^*]^*$ and $[q^*, p^*]^*$ are the same: use symmetry of $[\cdot, \cdot]_v$ if $\pi(p^*) \in
H\cdot \mathcal V$; use a similar argument for~$\pi(q^*)$. 
So $[\cdot, \cdot]^*$ is symmetric.

\medskip
\noindent Set $P = [\pi(p^*), \pi(r^*)] \cap [\pi(p^*), \pi(q^*)]$ and $R = [\pi(p^*), \pi(r^*)] \cap [\pi(q^*), \pi(r^*)]$.
By thinness of~$[\cdot, \cdot]$, $[\pi(p^*), \pi(r^*)] = P \cup R$.
By arc-overlaps for $[\cdot, \cdot]$,~$P$ and~$R$ are arcs (if nondegenerate).
As $(T, [\cdot, \cdot])$ is real, closed intervals in~$T$ are complete and $P, R$ are intervals.

Suppose $M = P \cap R$ is empty, then (without loss of generality) $P = [\pi(p^*), m)$ and $R = [m, \pi(r^*)]$ for some $m \in [\pi(p^*), \pi(r^*)]$.
So $[\pi(p^*), \pi(q^*)] \cap [\pi(q^*), \pi(r^*)] = [\pi(q^*), m)$ and $[\pi(p^*), \pi(q^*)] = [\pi(p^*), m) \cup (m, \pi(q^*)]$ --- absurd;
therefore, $M$ is not empty.

\noindent In fact, $M = \{ m \}$ --- pseudocenters for $[\cdot, \cdot]$.
Thus $P = [\pi(p^*), m]$ and $R = [m, \pi(r^*)]$.


\medskip
\noindent \textbullet ~Thin (axiom~\ref{axiom - thin}):
 
Suppose $\pi(p^*) = \pi(r^*)$. If $\pi(p^*) \notin H \cdot \mathcal V$, then $p^* = r^*$ and the axiom holds trivially. 
Otherwise, $[p^*, r^*]^* = \iota(y, [p,r]_v)$.
The $\pi$-preimages of $\pi(p^*)$ in $[p^*, q^*]^*$ and $[q^*, r^*]^*$ are $\iota(y, [p,q]_v)$ and $\iota(y,  [q,r]_v)$ respectively, and the axiom follows from thinness of $[\cdot, \cdot]_v$.

Now assume $\pi(p^*) \neq \pi(r^*)$.
By definition of $[\cdot, \cdot]^*$, the $\pi$-preimages of $[\pi(p^*), m)$ in $[p^*,r^*]^*$ and $[p^*, q^*]^*$ are the same subset of~$T^*$ as they depend only on the $\mathcal V$-turns in $P$. 
The same goes for $(m, \pi(q^*)]$. 
Thus \[ [p^*, r^*]^* \setminus \pi^{-1}(m) \subset \left([p^*, q^*]^* \cup [q^*, r^*]^*\right) \setminus \pi^{-1}(m).\]
To get the desired conclusion, we need only check that \[ [p^*, r^*]^* \cap \pi^{-1}(m) \subset \left([p^*, q^*]^* \cup [q^*, r^*]^*\right) \cap \pi^{-1}(m).\]
If $m \notin (F \cdot \mathcal V)$, then $\pi^{-1}(m)$ is a singleton and we are done.
Otherwise, the $\pi$-preimages of~$m$ in $[p^*,r^*]^*$,  $[p^*,q^*]^*$, and  $[q^*,r^*]^*$ are $\iota(y, [p,r]_v )$,  $\iota(y, [p,q]_v )$, and $\iota(y, [q,r]_v )$ respectively.
The axiom follows from thinness of~$[\cdot, \cdot]_v$.

\medskip
\noindent \textbullet ~Linear (axiom~\ref{axiom - linear}): 

Suppose $r^* \in [p^*, q^*]^*$ and $q^* \in [p^*, r^*]$. Then $\pi(q^*) = \pi(r^*)$ by linearity of $[\cdot, \cdot]$. 
If $\pi(q^*) \notin H \cdot \mathcal V$, then $q^* = r^*$ and we are done.
Otherwise, $q^* = \iota(y,q)$ and $r^* = \iota(y,r)$.
The $\pi$-primages of $\pi(q^*)$ in $[p^*, q^*]^*$ and $[p^*, r^*]^*$ are $\iota(y, [p,q]_v)$ and $\iota(y, [p,r]_v)$ respectively, and the axiom follows from linearity of $[\cdot, \cdot]_v$.
\end{proof}

\hypertarget{thm - ideal stitch}{}
\bigskip
\begin{proof}[Characterizing when~$f^*$ is a pretree-automorphism in Theorem~\ref{thm - ideal stitch}] {~}

Recall that the set-bijection~$f^*\colon T^* \to T^*$ is a pretree-automorphism exactly when $f^*([p^*,q^*]^*) = [f^*(p^*), f^*(q^*)]^*$ for all $p^*, q^* \in T^*$. 
Let $\pi\colon T^* \to T$ be the collapse map from the construction of~$T^*$.
By definition of~$f^*$, $\pi \circ f^* = f \circ \pi$.
As~$\pi$ restricts to a pretree-isomorphism $\iota(T \setminus (F \cdot \mathcal V)) \to T \setminus (F \cdot \mathcal V)$, we only need to check that 
\[f^*([p^*,q^*]^* \cap \iota(F \times \mathcal T_{\mathcal V})) =  [f^*(p^*), f^*(q^*)]^*  \cap \iota(F \times \mathcal T_{\mathcal V}) \]
Moreover, if~$\pi(p^*) = \pi(q^*) = x \cdot v$ for some $x \in F$ and $v \in \mathcal V$, i.e.~$p^* = \iota(x, p)$ and $q^* = \iota(x, q)$, then 
\[\begin{aligned}
f^*( [p^*,q^*]^* ) &= f^*( \iota(x, [p,q]_v) ) \\
&= \iota(\psi(x)x_v, \hat h_v([p,q]_v) ) \\
&= [\,\iota(\psi(x)x_v, \hat h_v(p) ),~\iota(\psi(x)x_v, \hat h_v(q) )\,]^* = [f^*(p^*), f^*(q^*)]^*.
\end{aligned}\]
The first interesting case is when both $\pi(p^*)=p, \pi(q^*)=q$ are not in $F \cdot \mathcal V$, i.e.~$p^* = \iota(p)$ and $q^* = \iota(q)$.
We may also assume $p \neq q$ to ignore a degenerate case.
Since we are considering the intersection with $\iota(F \times \mathcal T_{\mathcal V})$, assume $y \cdot v \in [p,q]$.
Let $ys_{-} \cdot d_{-}$ and $ys_{+} \cdot d_{+}$ (for some $s_{\pm} \in G_v$, $d_{\pm} \in \mathcal D_v$) be the two directions at $y\cdot v$ determined by the interval $[p,q]$.
So \[\begin{aligned} f^*([p^*, q^*]^* \cap \iota(y, T_v)) &= f^*(\iota(y, [s_{-} \cdot p_{-}, s_{+} \cdot p_{+}]_v )) \\
&= \iota(\psi(y)x_v, [\psi_v(s_{-}) \cdot \hat h_v(p_{-}), \psi_v(s_{+}) \cdot \hat h_v(p_{+})]_{\beta \cdot v}),
\end{aligned}\]
where $p_{\pm}$ are the attaching points in $\vec p$ chosen for direction $d_{\pm}$ respectively.

Applying~$f$ to $y\cdot v \in [p,q]$ gives us $\psi(y)x_v \cdot \beta(v) \in [f(p), f(q)]$;
recall that $f(v) = x_v \cdot \beta(v)$ by definition of~$\beta \in \mathrm{Sym}(\mathcal V)$.
The two directions at $\psi(y)x_v \cdot \beta(v)$ determined by the interval $[f(p),f(q)]$ are $\psi(ys_{\pm})x_v s_{\partial \cdot d_{\pm}} \cdot \partial(d_{\pm})$ (for some $s_{\partial \cdot d_{\pm}} \in G_{\beta \cdot v}$) respectively by definition of~$\partial \in \mathrm{Sym}(\bigcup_{v \in \mathcal V} \mathcal D_v)$.
Since $\psi(ys_{\pm})x_v= \psi(y)x_v\psi_v(s_{\pm})$ and $\psi_v(s_{\pm}) \in G_{\beta \cdot v}$, we have
\[ [f^*(p^*), f^*(q^*)]^*  \cap \iota(\psi(y)x_v, T_{\beta \cdot v}) = \iota(\psi(y)x_v,  [\psi_v(s_{-})s_{\partial \cdot d_{-}} \cdot q_{-}, \psi_v(s_{+})s_{\partial \cdot d_{+}} \cdot q_{+}]_{\beta \cdot v}), \]
where $q_{\pm}$ are the attaching points in $\vec p$ chosen for directions $\partial(d_{\pm})$ respectively.

So $f^*([p^*, q^*]^* \cap \iota(y, T_v)) =  [f^*(p^*), f^*(q^*)]^*  \cap \iota(\psi(y)x_v, T_{\beta \cdot v})$ if and only if \[ \hat h_v(p_{-}) = s_{\partial \cdot d_{-}} \cdot q_{-} \quad \text{and} \quad \hat h_v(p_{+}) = s_{\partial \cdot d_{+}} \cdot q_{+}.\]

\medskip
The remaining cases are handled almost exactly the same. 
Put together, this proves that the set-bijection~$f^*$ is a pretree-automorphism if and only if 
\[ \hat h_v(p_d) = s_{\partial \cdot d} \cdot p_{\partial \cdot d} \quad \text{for all } v \in \mathcal V \text{ and } d \in \mathcal D_v. \qedhere\]
\end{proof}

\bibliography{zrefs}
\bibliographystyle{plain}

\end{document}